\documentclass[a4paper,oneside,11pt,reqno]{amsart}
\usepackage[margin=2.8cm, marginpar=2.3cm]{geometry}
\usepackage[utf8]{inputenc}
\usepackage{stmaryrd}
\usepackage{bbm}
\usepackage[dvipsnames]{xcolor}
\usepackage{mathrsfs}
\usepackage{enumerate}
\usepackage{centernot}
\usepackage{latexsym,amsxtra}
\usepackage{dsfont}
\usepackage{slashed}
\usepackage[all]{xy}
\usepackage{amscd,graphics}
\usepackage{amsmath,amsfonts,amsthm,amssymb}
\usepackage{latexsym,amsmath}
\usepackage{graphicx,psfrag}
\usepackage{mathabx}
\usepackage{hyperref}
\usepackage{fancyhdr}
\usepackage[color=blue!20!white,textsize=tiny]{todonotes}
\usepackage{xcolor}

\usepackage{cleveref}
\usepackage{comment} 
\usepackage{braket}
\usepackage{caption}
\usepackage{enumitem}

\theoremstyle{plain}
\newtheorem{theo}{Theorem}[section]
\crefname{theo}{Theorem}{Theorems}
\Crefname{theo}{Theorem}{Theorems}
\newtheorem{prop}[theo]{Proposition}
\crefname{prop}{Proposition}{Propositions}
\Crefname{prop}{Proposition}{Propositions}
\newtheorem{lem}[theo]{Lemma}
\crefname{lem}{Lemma}{Lemmas}
\Crefname{lem}{Lemma}{Lemmas}

\crefname{conj}{Conjecture}{Conjectures}
\Crefname{conj}{Conjecture}{Conjectures}
\newtheorem{cor}[theo]{Corollary}
\crefname{cor}{Corollary}{Corollaries}
\Crefname{cor}{Corollary}{Corollaries}

\crefname{claim}{Claim}{Claims}
\Crefname{claim}{Claim}{Claims}

\crefname{property}{Property}{Properties}
\Crefname{property}{Property}{Properties}

\crefname{problem}{Problem}{Problems}
\Crefname{problem}{Problem}{Problems}

\theoremstyle{definition}
\newtheorem{defi}[theo]{Definition}
\crefname{defi}{Definition}{Definitions}
\Crefname{defi}{Definition}{Definitions}

\crefname{notation}{Notation}{Notations}
\Crefname{notation}{Notation}{Notations}

\crefname{convention}{Convention}{Conventions}
\Crefname{convention}{Convention}{Conventions}

\crefname{cond}{Condition}{Conditions}
\Crefname{cond}{Condition}{Conditions}
\newtheorem{assum}[theo]{Assumption}
\crefname{assum}{Assumption}{Assumptions}
\Crefname{assum}{Assumption}{Assumptions}
\newtheorem{ex}{Example}[section]
\crefname{ex}{Example}{Examples}
\Crefname{ex}{Example}{Examples}

\theoremstyle{remark}
\newtheorem{rem}[theo]{Remark}
\crefname{rem}{Remark}{Remarks}
\Crefname{rem}{Remark}{Remarks}
\theoremstyle{question}

\crefname{ques}{Question}{Questions}
\Crefname{ques}{Question}{Questions}

\crefname{section}{Section}{Sections}
\Crefname{section}{Section}{Sections}
\crefname{subsection}{Subsection}{Subsections}
\Crefname{subsection}{Subsection}{Subsections}
\crefname{figure}{Figure}{Figures}
\Crefname{figure}{Figure}{Figures}

\newcommand{\Z}{\mathbb{Z}}

\newcommand{\R}{\mathbb{R}}
\newcommand{\C}{\mathbb{C}}

\newcommand{\pt}{\mathrm{pt}}

\newcommand{\diag}{\mathrm{diag}}

\newcommand{\fraks}{\mathfrak{s}}
\newcommand{\frakt}{\mathfrak{t}}

\newcommand{\SW}{\mathrm{SW}}

\newcommand{\id}{\mathrm{id}}

\newcommand{\wt}{\widetilde}

\def\C{\mathbb{C}}

\def\spinc{\text{spin}^c}
\def\wt{\widetilde}

\newcommand{\Diff}{\mathrm{Diff}}

\newcommand{\del}{\partial}

\newcommand{\rank}{\mathop{\mathrm{rank}}\nolimits}

\newcommand{\Th}{\mathop{\mathrm{Th}}\nolimits}

\newcommand{\F}{\mathbb{F}}
\newcommand{\s}{\mathfrak{s}}
\newcommand{\FSW}{\mathrm{FSW}}

\title[exotic embeddings from diagonalization]{Exotically knotted closed surfaces from Donaldson's diagonalization for families}

\begin{document}

\author{Hokuto Konno}
\address{Graduate School of Mathematical Sciences, the University of Tokyo, 3-8-1 Komaba, Meguro, Tokyo 153-8914, Japan}
\email{konno@ms.u-tokyo.ac.jp}

\author{Abhishek Mallick}
\address{Department of Mathematics, Rutgers University, Hill Center, Busch Campus, 110 Frelinghuysen Road
Piscataway, NJ 08854, USA}
\email{abhishek.mallick@rutgers.edu}

\author{Masaki Taniguchi} 
\address{Department of Mathematics, Graduate School of Science, Kyoto University, Kitashirakawa Oiwake-cho, Sakyo-ku, Kyoto 606-8502, Japan}
\email{taniguchi.masaki.7m@kyoto-u.ac.jp}

\maketitle

\begin{abstract}
We introduce a method to detect exotic surfaces without explicitly using a smooth 4-manifold invariant or an invariant of a 4-manifold-surface pair in the construction.  Our main tools are two versions of families (Seiberg--Witten) generalizations of Donaldson's diagonalization theorem, including a real and families version of the diagonalization. This leads to an example of a pair of exotically knotted $\mathbb{R}P^2$'s embedded in a closed 4-manifold whose complements are diffeomorphic, making it the first example of a non-orientable surface with this property. In particular, any invariant of a 4-manifold-surface pair (including invariants from real Seiberg--Witten theory such as Miyazawa's invariant) fails to detect such an exotic $\mathbb{R} P^2$. 
One consequence of our construction reveals that \textit{non}-effective embeddings of corks can still be useful in pursuit of exotica. Precisely, starting with an embedding of a cork $C$ in certain a 4-manifold $X$ where the cork-twist does \textit{not} change the diffeomorphism type of $X$, we give a construction that provides examples of exotically knotted spheres and $\mathbb{R}P^2$'s with diffeomorphic complements in $ C \# S^2 \times S^2 \subset X \# S^2 \times S^2$ or $C \# \mathbb{C}P^2  \subset X \# \mathbb{C}P^2 $. In another direction, we provide infinitely many exotically knotted embeddings of orientable surfaces, closed surface links, and 3-spheres with diffeomorphic complements in once stabilized corks, and show some of these surfaces survive arbitrarily many internal stabilizations. By combining similar methods with Gabai's 4D light-bulb theorem, we also exhibit arbitrarily large difference between algebraic and geometric intersections of certain family of 2-spheres, embedded in a 4-manifold.

\end{abstract}


\section{Introduction}
 Finashin, Kreck, and Viro \cite{FKV87} gave the first example of a pair of exotic surface by showing the existence of an exotic $\#_{10} \mathbb{R}P^2$ in $S^4$. Since then the study of exotic surfaces has evolved to be one of the central topic of interest in low-dimensional topology. Indeed, a major open question in the field is to construct a pair of closed orientable and exotically knotted surface in $S^4$ (or in 4-manifolds as small as possible). Here, by exotic surfaces, we refer to a pair of closed surfaces $S_1$ and $S_2$ embedded in a 4-manifold $X$ such that $S_1$ is topologically isotopic (rel $\partial X$) to $S_2$, but they are not smoothly isotopic. Oftentimes in the literature, in lieu of obstructing the smooth isotopy directly, one obstructs the existence of a diffeomorphism of $X$ taking $S_1$ to $S_2$. Borrowing terminology from \cite{Fi02}, we refer to such a pair of surfaces as \textit{non-equivalent}, with the understanding that if the surfaces were related by a diffeomorphism of $X$, we shall refer to it them as \textit{equivalent} surfaces.

There are many instances in the literature where smooth 4-manifold invariants such as the Seiberg--Witten invariants have been used to distinguish pairs of non-equivalent surfaces. One of the widely used strategies for such obstruction has been to use the Fintushel-Stern rim surgery construction  \cite{FS97} together with the formula for the Seiberg--Witten invariant under the surgery. For instance, in \cite{Fi02} Finashin used a similar strategy to produce non-equivalent surfaces in $\mathbb{C}P^2$ (these surfaces were shown to be topologically isotopic, hence exotic, by Kim-Ruberman \cite{kim2008smooth}). Several other articles have followed a similar strategy, see for example Hayden \cite{Ha20}. Other approaches arose from the work of Akbulut \cite{akbulut2014isotoping}, and Auckly-Kim-Melvin-Ruberman \cite{AKHMR15} and Baykur-Hayano \cite{baykur2016multisections}.

The unoriented version of this question also has been well-studied in the literature which again began with the work of \cite{FKV87} producing non-equivalent $\#_{10} \mathbb{R}P^2$ in $S^4$. Later Finashin-Kreck-Viro \cite{Fi06}, Finashin \cite{Fi07}, Havens \cite{Ha21}, Levine-Lidman-Piccirillo \cite{LLP23} and Mati\'c-\"Ozt\"urk-Reyes-Stipsicz-Urz\'ua\cite{matic2023exotic} made progress in this realm.

\noindent
Notably, a recent interesting result of Miyazawa \cite{Mi23}, shows that the existence of a pair of a non-equivalent $\mathbb{R}P^2$ embedded in $S^4$ using real Seiberg--Witten theory. This construction relies on an invariant of the 4-manifold-surface pair.

Although we exclusively focus on closed surfaces, several studies have also produced exotic surfaces bounded by knots in $B^4$, these include the work of  Hayden-Sundberg \cite{KI24}, using Khovanov homology, Juh\'asz-Miller-Zemke \cite{JMZ21} using twisted link Floer homology, and Dai-Stoffregen and the second author \cite{DMS23} using knot Floer homology (see also \cite{hayden2021corks, DHM22}).

All of the studies mentioned above for closed exotic surface detection in the literature in fact factor through proving that some operation made to the surfaces which results into non-diffeomorphic 4-manifolds (or the 4-manifold-surface pair). Such operations include blowing-down spheres, taking the double-branched covers, using various surgery operations and so on, coupled with a suitable smooth 4-manifold invariant which is used to distinguish the resulting 4-manifolds.

In particular, the possible existence of exotic surfaces that are diffeomorphic (as a 4-manifold-surface pair) remains to the blind spot of these techniques. It is then natural to ask when it is possible to produce exotic surfaces that are in fact related by a diffeomorphism of the closed 4-manifold, but there is no smooth isotopy of the 4-manifold that takes one surface to another, i.e. exotic surfaces that are \textit{equivalent}. In order to distinguish a pair of equivalent exotic surfaces, 4-manifold invariants are typically not that useful. 

The other motivation for studying equivalent surfaces stem from the Cerf's theorem \cite{cerf2006diffeomorphismes} which says that any two embeddings of $S^1$ in $S^3$, that are related by a diffeomorphism are in fact smoothly isotopic. Hence it is natural to contrast this behavior with the 4-dimensional case. 

The first systematic study of equivalent but exotic surfaces was given by the work of Baraglia \cite{Ba20} using the adjunction inequality from family Seiberg--Witten thoery. Later Auckly \cite{Au23} also used family adjunction inequality to produce closed equivalent exotic surfaces that survive many internal stabilizations. In particular, in \cite{Ba20} Baraglia showed the existence of equivalent exotic surfaces in $\#_n K3  \#_n S^2 \times S^2$, for any $n >0$. Later in \cite{Au23}, Auckly proved a similar statement for any 4-manifold $X$ stabilized by a number of $S^2 \times S^2$'s (where the number of such stabilization depends on $X$). However, all of these exotic surface detection results in the literature (equivalent or otherwise) requires the existence of a basic class of a 4-manifold.  Hence the 4-manifolds where such equivalent exotic surfaces were produced in \cite{Ba20, Au23} (see also \cite{schwartz2019equivalent}) had rather complicated intersection forms.

The primary aim of this article is to introduce a new tool for exotic surface detection without using any 4-manifold invariant or any invariant of the 4-manifold-surface pair. In fact, this procedure yields exotic surfaces that are equivalent. To achieve this, we use constraints obtained as certain generalizations of Donaldson's diagonalization theorem. This tool stems from a 1-parameter family version of real Seiberg--Witten theory. This leads to the first example of a non-orientable equivalent exotic manifold in a closed 4-manifold. We also use 2-parameter family version of Seiberg--Witten theory to obtain equivalent exotic spheres in a closed 4-manifold. The underlying geometric construction of our work also uses non-effective embeddings of corks \cite{Ak91_cork} in an interesting manner. In particular, since we do not use any exotic smooth structures in our construction, the 4-manifolds in which our exotic surfaces lie has comparatively simpler intersection form.

In another direction, we use family adjunction inequality to produce a wide range of equivalent exotic embeddings into a small 4-manifolds, such as a once stabilized contractible manifold. These embeddings include closed surfaces which survive many internal stabilizations. We also produce exotic surface links and exotic embeddings of $S^3$ in such a manifold. Using a similar approach, we also capture that the difference between algebraic and geometric intersection points of certain embedded 2-spheres in once stabilized contractible manifolds, can be arbitrarily large. Let us now state our results in more detail. 
\subsection{Exotic surfaces and Donaldson's diagonalization}
We introduce two new tools for exotic surface detection, which can be broadly characterized as generalizations of Donaldson's diagonalization theorem in an appropriate sense.
\subsubsection{\bf{Using the real family Seiberg--Witten theory}}
We start by formalizing a few notions by means of definitions which would be helpful for the statement of our theorems. 
In this paper, we discuss embedding of surfaces $i : S \to X$ into a 4-manifold $X$.
However, we often drop the notation of embedding $i$ and write $S \subset X$ by abuse of notation. This remark applies also to embeddings of 3-manifolds.

\begin{defi}\label{defi:stro_equi_closed}
We say that a pair of smoothly embedded closed surfaces $i: S\to X$ and $i': S\to X$ in a closed smooth 4-manifold $X$ are {\it equivalent} if there is an orientation preserving diffeomorphism $f: X \to X$ such that $f \circ i =i'$.  
\end{defi}
\noindent
From now on, we abbreviate the notation $i : S \to X$ of an embedding and simply write $S$ for the embedding. 
In this paper, we also consider closed surfaces embedded in a 4-manifold $X$ with boundary. Henceforth, we assume that surfaces are embedded in the interior of $X$, which, of course, can be achieved by isotopy without loss of generality. For 4-manifolds with boundary, we have natural two notions of equivalences.  
\begin{defi}\label{defi:strong_weak_equi_boundary}
Let $X$ be a smooth manifold with boundary.
A pair of smoothly embedded closed surfaces $S$ and $S'$ in $X$ are said to be {\it strongly equivalent via} $f$,  if there is a diffeomorphism $f : X \to X $ that satisfies $f(S)=S'$ and $f|_{\del X} = \id_{\del X}$.
If there is a diffeomorphism $f : X \to X$ that does not necessarily restrict to the identity of $\del X$, we say that $S$ and $S'$ are {\it weakly equivalent via} $f$. When the diffeomorphism $f$ is clear from the context, we will omit referring to it and simply regard the surfaces as strongly (weakly) equivalent. 
\end{defi}
Let us also formalize the definition of exotically knotted surfaces embedded in 4-manifolds with boundary.

\begin{defi}
\label{defi: relatively exotic}
Let $X$ be a smooth manifold with boundary.
A pair of smoothly embedded closed surfaces $S$ and $S'$ is said to be {\it relatively exotic} if there is a topological isotopy from $S$ to $S'$ relative to $\del X$, but there is no smooth isotopy from $S$ to $S'$ relative to $\del X$.
\end{defi}

We now state our first result:

\begin{theo}\label{RP2emb1}
There is an exotic pair of embeddings of $\R P^2$'s into a simply-connected closed smooth 4-manifold which are equivalent. 
\end{theo}

As discussed earlier, this constitutes the first example of a non-orientable exotically knotted surface with diffeomorphic complements, embedded in a 4-manifold. Let us now furnish a quick comparison of the tools we used with the existing ones in the literature:

In order prove Theorem~\ref{RP2emb1}, we use a generalization of considering a variation of Donaldson's diagonalization for manifold with boundary by Fr{\o}yshov \cite{Fr96}, in the context of \textit{real} family Seiberg--Witten theory. It can be regarded as a family version of \cite[Theorem 1.1]{KMT24i}. Of course, any method in the literature that only obstructs the existence of equivalent surfaces is also not suited of deducing Theorem~\ref{RP2emb1}.
In addition, the smallest such 4-manifold $X$ from Theorem~\ref{RP2emb1} has intersection form of $2(+1) \oplus 11 (-1)$. 

\begin{rem}
All known invariants of 4-manifold-surface pairs cannot detect equivalent exotic $\R P^2$ embeddings since the invariants are obviously unchanged under equivalence by definition. In particular, any invariant from real Seiberg--Witten theory including Miyazawa's invariant \cite[Theorem 1.8]{Mi23} and $O(2)$--equivariant Bauer--Furuta invariant \cite{BH24} cannot detect this exotic pair. 
\end{rem}

\subsubsection{\bf{Using corks but not exotic smooth structures}}
Let us now draw readers' attention to a particularly interesting geometric aspect of the proof of Theorem~\ref{RP2emb1}. To put this into context, we first recall a rather well-known connection between corks and exotic (non-equivalent) closed surface in a closed 4-manifold, known to the experts. This is due to \cite{AKHMR15} (see also \cite{akbulut2014isotoping}).

Let $C$ be a cork, that is $C$ a contractible manifold with boundary. $\partial C$ is further equipped with a diffeomorphism $\tau$ such that $\tau$ does \textit{not} extend over $C$ as a diffeomorphism (see Section~\ref{section:Preliminaries} for details). Suppose that there is an embedding of $C$ in a closed 4-manifold $X$. Now let us call the 4-manifold obtained from regluing $C$ using $\tau$ in $X$ by $X^{\tau}$ (which may not be diffeomorphic to $X$). Let us record the following definition, which we will repeatedly use:
\begin{defi}
\label{defi: weakly effective cork}
We call a pair $(C,\tau)$ an \textit{effective simple cork} of a 4-manifold $X$ if the following conditions are satisfied:
\begin{itemize}
    \item [(i)] Simple--The cork-twist dissolves after 1-stabilization: $\tau$ extends to $C\#S^2\times S^2$ (or $C \# \mathbb{C}P^2$) as a diffeomorphism
    \[
    \tilde{\tau}: C \# S^2\times S^2 \rightarrow C \# S^2 \times S^2,\; \tilde{\tau}|_\partial=\tau.
    \]
    \item [(ii)] Effective--There is an embedding of $C$ to $X$ so that regluing $C$ by the boundary twist changes the diffeomorphism type of $X$.
\end{itemize}
On the other hand, if a simple cork $(C,\tau)$ embeds into a 4-manifold $X$, where the cork twist does not change the diffeomorphism type of $X$, we call it a \textit{simple, non-effective cork}\footnote{This definition depends on the choice of the embedding of the cork, but we suppress it.} of $X$.
\end{defi} 
The definition of simple corks requires that we specify what we are stabilizing $C$ by ($S^2 \times S^2$ or $\C P^2$), but we will omit making it explicit as it will clear from the context.

In \cite{AKHMR15}, the authors considered an effective simple cork $(C,\tau)$ in a 4-manifold. Then they showed that there is pair of spheres embedded in $C \# \mathbb{C}P^2$ that are related by $\tilde{\tau}$ which constitute an exotic pair of (non-equivalent) surface in $X$. In order to achieve this, they showed that blowing down spheres (they are of self-intersection $(+1)$) changes $X \# \mathbb{C}P^2$ to $X$, and $X^{\tau}$ respectively. Hence the spheres could not have been related by a diffeomorphism. See Figure~\ref{fig: intro_cork_blowdown} for a schematic description
We make one more definition regarding the this procedure.
\begin{defi}\label{defi: related_surface}
$(C,\tau)$ be a simple cork. We say that a pair of closed surfaces $(S_1, S_2)$ are \textit{associated to} $(C, \tau)$ if they are embedded in $C \# S^2 \times S^2$ (or $C \# \mathbb{C}P^2$) and $\tilde{\tau}(S_1)=S_2$.
\end{defi}
In light of the above Definition ~\ref{defi: related_surface}, the discussion above (from \cite{AKHMR15}) shows the existence of a pair of exotic surfaces that are associated to a simple effective cork $(C,\tau)$ for $X$. Indeed, in their argument, it was crucial that the associated cork has an effective embedding in $X$. Now consider the situation where there is an non-effective embedding of a simple cork $C$ is a closed 4-manifold $X$. It is now quite unclear how to produce exotic surfaces associated with the embedding of $(C,\tau)$ in $X$, since there is no underlying exotic smooth structure of $X$ to exploit. In our proof of Theorem~\ref{RP2emb1}, the (equivalent) exotic surfaces that we construct are exactly of this type.

\begin{figure}[h!]
\center
\includegraphics[scale=0.6]{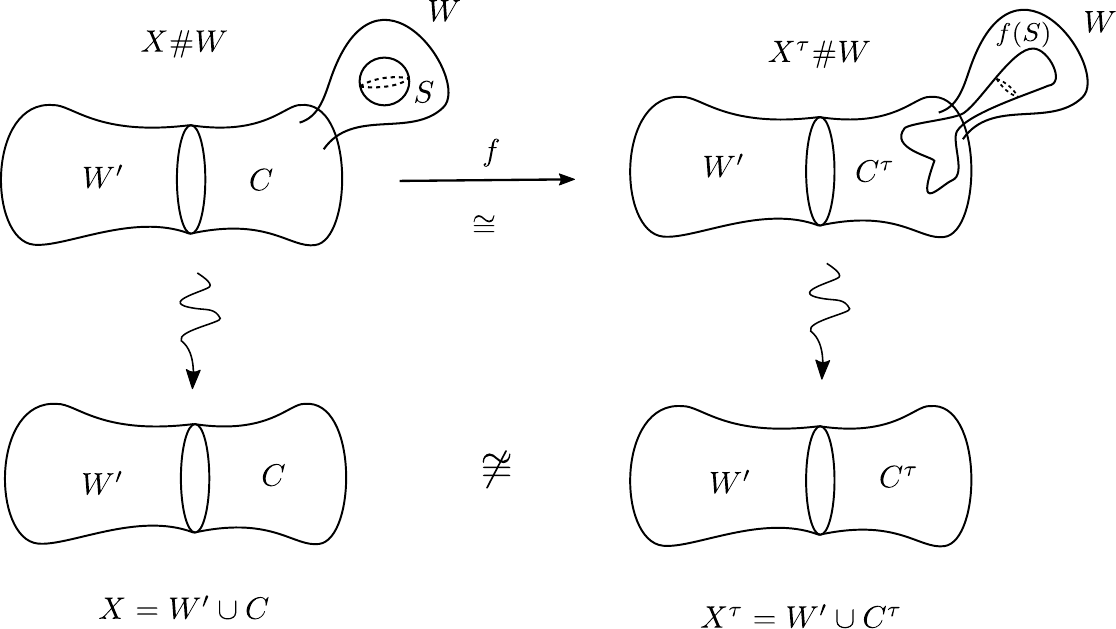}
\caption{\raggedright A schematic: producing exotic surfaces using the exotic smooth structures. Here $W$ is either $\mathbb{C}P^2$ or $S^2 \times S^2$, $f$ is a diffeomorphism of the stabilized manifolds and the wiggly arrows indicate blowing-down or surgery operation. It is however unclear how to produce exotic surfaces when $X$ and $X^{\tau}$ are diffeomorphic, using a similar strategy. Contrast with Theorem~\ref{intro: thm RP2 local}.}\label{fig: intro_cork_blowdown}
\end{figure}

\begin{theo}\label{intro: thm RP2 local}
There exist a pair of exotic but equivalent $\mathbb{R}P^2$, embedded in $X \# \mathbb{C}P^2$, associated to a simple non-effective cork of a closed simply-connected 4-manifold $X$.
\end{theo}
\noindent
In Theorem~\ref{intro: thm RP2 local}, the 4-manifold can be taken to be homeomorphic to $\C P^2 \# 11 \overline{\C P}^2$ and the cork can be taken as the Akbulut cork. The topological isotopy in Theorem~\ref{intro: thm RP2 local} is aided by the result of Orson and Powell \cite{OP22}.

\begin{rem}
    In the proof of Theorem~\ref{intro: thm RP2 local}, the contribution of the cork towards the exotic surface detection comes from the interplay of the cork-twist extension and other diffeomorphisms. In particular, this unveils a new  approach through which corks can be useful in pursuing exotica. As far as the authors are aware, in all the previous instances in the literature where corks has been useful in detecting an exotic phenomena (other than the exotic smooth structure), such as the exotic surface  (as described above), or exotic diffeomorphism \cite{konno2024localizing} detection, it was crucial that cork-twist changed the smooth structure of a certain closed 4-manifold used in the construction.
\end{rem}  

\subsubsection{\bf{Using 2-parameter family Seiberg--Witten theory and Seidel Dehn twists}}
While the examples in the previous subsection consist of non-orientable surfaces, we also produce exotic equivalent spheres, while still using non-effective embeddings of corks.

\begin{theo}\label{intro: s1_family_surface_closed}
There exists a simply-connected closed smooth 4-manifold $X$ containing a pair of equivalent exotic spheres representing the homology class
    \[
    [S^2\times \{*\} ] + [\{* \} \times S^2 ] \in H_2(X \# S^2\times S^2)
    \]
associated to a simple non-effective cork embedded in $X \# S^2 \times S^2$.
\end{theo}
The minimal examples of 4-manifolds $X$ that appear in \cref{intro: s1_family_surface_closed} are homeomorphic to $\C P^2 \# 11 \overline{\C P}^2$, and the associated cork here again can be taken to be the Akbulut cork. Note that such an exotic sphere cannot be detected using the families adjunction inequality \cite{Ba20}, because of $b^+(X)<3$ so that the family Seiberg-Witten invariant of a diffeomorphism is not well-defined. In order to prove Theorem~\ref{intro: s1_family_surface_closed}, we use 2-parameter family Seiberg--Witten theory obstructions developed by the authors in \cite{KMT23}, which again can be thought of as a generalization of the Donaldson's diagonalization theorem in an appropriate context. 

Another geometric aspect involved in the proof of Theorem~\ref{intro: s1_family_surface_closed} is the \textit{Seidel Dehn twists} (which is also sometimes referred to as the {\it reflection on sphere} in the literature). Given a smoothly embedded $(-2)$-sphere $S$ in an oriented 4-manifold $X$, one can get a diffeomorphism $T_S : X \to X$ which is supported in a neighborhood of $S$ and acts as the antipodal map on $S$, which is called the {\it Dehn twist about $S$} (see the proof of Theorem~\ref{intro: s1_family_surface_closed} for details). When $X$ is symplectic and $S$ is Lagrangian, the Dehn twist about $S$ produces a symplectomorphism and has been extensively studied after Seidel's thesis \cite{Sei08}.

Let us now clarify the part played by corks in both Theorem~\ref{intro: s1_family_surface_closed} and \ref{intro: thm RP2 local}. In fact, this will also speak towards the user-friendly nature of the obstruction. 

\begin{theo}\label{intro: main_lemma}
Let $(C , \tau)$ be a simple cork. Suppose that $\partial C$ bounds a smooth compact spin$^c$ 4-manifold $(X, \mathfrak{s})$ with $b^+(X)=1$ and $b_1(X)=0$. 
Now if $\Phi$ is any orientation-preserving, smooth extension of $\tau$ so that $\Phi$ is $H^{+}$-reversing, and $\spinc$-preserving i.e. $f^{*}(\s)=\s$ and we have,
\[
c_1(\fraks)^2 > \sigma(X),
\]
then there is a pair of exotic equivalent $\mathbb{R}P^2$ (or $S^2$) associated to a non-effective\footnote{By construction $C$ is a non-effective cork in $X \cup C$.} embedding of the cork $C$ in $X \cup C \# \mathbb{C}P^2 $ (or in $X \cup C \# S^2 \times S^2$). In fact, this pair of exotic $\mathbb{R}P^2$ (or $S^2$) are weakly equivalent in $C \# \mathbb{C}P^2$ (or in $C \# S^2 \times S^2$). 
\end{theo}
Here, by $H^{+}$-reversing, we mean that $\Phi$ changes the homology orientation of $X$, (see  Section~\ref{section:Preliminaries} for details). Theorem~\ref{intro: main_lemma} is the main ingredient that is behind our construction. Note that in this case, in order to show the existence of the exotic surfaces, we only needed to compute the intersection form of the manifold $X$. In \cite{KMT23first}, the authors gave infinitely many examples of simple corks that satisfy the hypothesis of Theorem~\ref{intro: main_lemma}. For reference, we refer to simple corks that satisfy the hypothesis of Theorem~\ref{intro: main_lemma} as \textit{corks detected by the $\spinc$-preserving family Seiberg--Witten theory}, see Section~\ref{section:Preliminaries}.

\subsection{Strongly equivalent exotic surfaces in once stabilized contractible manifolds}
As discussed in the introduction, the smallest closed 4-manifold supporting a pair of an exotic equivalent surface was embedded in a 4-manifold with the second Betti number $25$ \cite{Ba20}. For manifold with boundary,  Iida, Mukherjee, the first and the third authors \cite[Remark 6.5]{IKMT22} provided an equivalent exotic pair of embedded spheres in a compact 4-manifold with boundary, whose second Betti number is $5$. In this article, we provide equivalent exotic spheres in a compact 4-manifold with boundary, whose second Betti number is $1$ and $2$. To express our results clearly, we will need one more definition:
\begin{defi}\label{defi:sw_eff_cork}
We say that a cork $(C,\tau)$ is \textit{SW-effective} if there exists a 4-manifold $X$, so that there is an (SW) effective embedding of the cork in the following sense: There is a closed smooth oriented simply-connected 4-manifold $X$ with $b^+(X)\geq2$, a spin$^c$ structure $\fraks$ on $X$ of $d(\fraks)=0$\footnote{Here $d(\fraks)$ represents the formal dimension of the Seiberg--Witten moduli space}, and there is a smooth embedding $i : C \hookrightarrow X$ such that the Seiberg--Witten invariants $\SW(X,\fraks)$ and $\SW(X^\tau, \fraks)$ are distinct in $\Z/2$, where $X^\tau$ is the cork-twisted copy of $X$ along $(C,\tau)$ via $i$.
\end{defi}

Note that begin SW-effective is a different notion than being effective (Definition~\ref{defi: weakly effective cork}). We prove the following result, which roughly shows that every non-zero and non-negative self-intersection homology class of $S^2 \times S^2$ admits infinitely many (pairwise) equivalent but exotic sphere representatives, in once stabilized corks.

\begin{theo}\label{theo: intro sphere from adjunction}
Let $C$ be a simple SW-effective cork and let $\alpha \in H_2(C\#S^2\times S^2;\Z) \setminus \{0\}$ be any homology class which has non-ngative self-intersection. Then there exist infinitely many oriented closed smoothly embedded surfaces $\{S_i\}_{i=1}^\infty$ in $C\#S^2\times S^2$ with $\alpha = [S_i]$ that are strongly equivalent but relatively exotic to each other.
\end{theo}
Theorem~\ref{theo: intro sphere from adjunction} in fact follows from a more general Lemma~\ref{lem: sphere from adjunction} and \ref{lem: surfaces from adjunction}, where we use a version of the adjunction inequality for families derived by Baraglia \cite{Ba20} to produce strongly equivalent yet exotic surfaces. 

Moreover, by using the recent work of Baraglia-Konno \cite{BK18} and Auckly \cite{Au23}, we show that it is possible to produce exotic surfaces in once stabilized contractible manifolds, so that they survive arbitrarily many internal stabilizations by $T^2$. This is particularly relevant in the light of the work of Baykur-Sunukjian \cite{baykur2016knotted}, where it is shown that any two exotic surfaces become smoothly isotopic after some number of internal stabilization by $T^2$.

\begin{theo}\label{thm: internal_stab_intro}
For any natural number $p$, there are infinitely many contractible smooth compact 4-manifolds $C_n$ such that there is a pair of strongly equivalent but relatively exotic spheres in $C_n \# S^2 \times S^2$ that survive $p$-many internal stabilization by $T^2$.
\end{theo}
We end this subsection by considering a consequence of our technique for producing exotic surfaces in once stabilized corks coupled with Gabai's 4D light bulb theorem \cite[Theorem 1.2]{Ga20}. Broadly, we detect the difference between the algebraic and geometric intersection of 2-spheres embedded in a 4-manifold.

Recall that a version of the generalized light bulb theorem states the following: For a simply-connected compact 4-manifold $W$ with boundary, any two spheres that are homotopic to each other while admitting a common dual sphere, are in fact smoothly isotopic. Here by a dual sphere $S'$ to a sphere $S$, we refer to a square zero sphere $S'$ that intersects $S$ geometrically once. We prove the following  estimates of geometric intersection numbers:
\begin{cor} \label{geometric int number}
Let $C$ be a simple SW-effective cork and $n$ be a non-negative integer.
Moreover, let $S$ be a smoothly embedded 2-sphere in $C\#S^2\times S^2$ which is obtained as the graph 
\[
S = \{(x, f_n(x)) \in S^2 \times S^2| x \in S^2\} \subset C\#S^2\times S^2
\]
of a smooth degree-$n$ map $f_n : S^2 \to S^2$. Then there is a smoothly embedded sphere $S'$ in $C \# S^2 \times S^2$ with the following properties: 

\begin{itemize}

\item[(i)] The 2-sphere $S'$ is topologically isotopic rel boundary and strongly equivalent to $S$.

\item [(ii)] For any smoothly embedded sphere $S''$ relatively and smoothly isotopic to $S'$, which is transverse to $\{* \} \times S^2$, we have  
\[
2 n +1\leq |S'' \cap (\{* \} \times S^2)|.
\]
\end{itemize}
\end{cor}
\noindent
Note that the algebraic intersection of $S''$ and $\{* \} \times S^2$ is one.

\subsection{Exotic $S^3$-embeddings}
Apart from exotically knotted embeddings of surfaces in 4-manifolds, there is another interesting class of exotic embeddings which stem from 3-manifolds. Indeed, recent work \cite{IKMT22, KMT22, KMT23first} shows the existence of such codimension one embeddings \footnote{As in the surface case, we use the terms of strong/weak equivalence and relative exotica for embeddings of 3-manifolds into 4-manifolds.}. 
However, constructing equivalent exotic embeddings of $S^3$ in a small 4-manifold is still a hard question. 
The smallest known example of a 4-manifold admitting such an embedding has $b_2 =4$, see \cite[Remark 1.8]{IKMT22}. We shall provide a 4-manifold containing equivalent exotic $S^3$ with $b_2=2$, namely one coming from once stabilized corks.

\begin{theo}
\label{theo: intro S3}
Let $C$ be a simple SW-effective cork. 
Let $Y$ be the separating $S^3$ in the neck of $C\#S^2\times S^2$.
Then there exist infinitely many smoothly embedded 3-spheres $\{Y_i\}_{i=1}^\infty$ in $C\#S^2\times S^2$ that are strongly equivalent to $Y$ but mutually relatively exotic.
\end{theo}

\subsection{Exotic surface links}
To further demonstrate the wide applicability of our methods to produce exotically knotted surfaces, we now exhibit examples of exotically knotted closed surfaces with multiple connected components. We call such a surface an \textit{exotically knotted closed surface link}. There is not much known about the existence of exotic surface links. In \cite[Theorem 8.1]{HKKMPS21}, Hayden-Kjuchukova-Krishna-Miller-Powell-Sunukjian showed the existence of an exotic surface link (with 2-components) in a closed 4-manifold. In particular, they distinguish the smooth isotopy between the surface links by obstructing a diffeomorphism of the 4-manifold-surface link pair. This is done by use of surgery and reducing the obstruction to a question of distinguishing knotted disks (with multiple components). 

In alignment with the theme of the article, we give an example of an exotic surface link (all of whose components are spheres) in a closed 4-manifold which are equivalent (that there is a diffeomorphism of the closed 4-manifold which takes the surface link to itself). In particular, it is not possible to distinguish the pair of exotic surface links by surgery operations and using existing obstructions of the diffeomorphism type of the complements. Moreover, we also show that each component of this surface link is in fact smoothly isotopic to the corresponding component of the link given by the diffeomorphism mediating the equivalency.

\begin{theo}\label{S2link}
There is an exotic pair of four component surface links 
\[
f, f' : S^2 \cup S^2 \cup S^2 \cup S^2 \hookrightarrow X
\]
in a simply-connected closed 4-manifold $X$ which are equivalent and each component is smoothly isotopic each other. 
\end{theo}

These examples are obtained from modifying our approach to produce exotic surfaces in 4-manifolds generalizing Donaldson's diagonalization for family real Seiberg--Witten theory. Moreover, these examples of the exotic surfaces, like Theorem~\ref{intro: thm RP2 local}, are associated to a non-effective embedding of a cork in a 4-manifold $X$. The 4-manifold in Theorem~\ref{S2link} is obtained by a three-fold stabilization of $X$.

\textbf{Organization}
In \cref{section:Preliminaries}, we review a class of strong corks that are detected by the family Seiberg--Witten theory. In \cref{section: A family version of real Seiberg--Witten theory}, we establish a real 1-parameter family version of Donaldson's diagonalization theorem, which serves as an obstruction to smooth isotopy between certain surfaces. In \cref{section: Diffeomorphisms out of corks}, we provide examples of simple (SW) effective corks. Finally, in \cref{section: Proofs of main theorems}, we present the proofs of the main theorems.

\textbf{Acknowledgement} We thank Maggie Miller for helpful conversations. We also thank Irving Dai, Kristen Hendricks and Jin Miyazawa for their comments on the article.
HK was partially supported by JSPS KAKENHI Grant Number 21K13785.
AM was partially supported by NSF DMS-2019396. MT was partially supported by JSPS KAKENHI Grant Number 20K22319 and 22K13921.

\section{Preliminaries: Family diagonalization theorem  over $S^1$ and corks}\label{section:Preliminaries}

We first explain the result \cite[Theorem 1.7]{KMT23first}, which showed that family Seiberg--Witten theory detects corks. Specifically, the obstruction comes from a version of the family diagonalization theorem over $S^1$. 

First, let us recall the definition of $H^+(X)$ for an oriented compact 4-manifold $X$.
Let $X$ be an oriented compact 4-manifold, with or without boundary.
The Grassmanian $\mathrm{Gr}^+(X)$, consisting of maximal-dimensional positive-definite (with respect to the intersection form) subspaces of $H^2(X;\R)$ is known to be contractible (see for example \cite[Subsection~3.1]{LL01}).
By $H^{+}(X)$, we denote a choice of such a subspace. Let $\phi : H^2(X;\R) \to H^2(X;\R)$ be an automorphism of the intersection form. Then since $\mathrm{Gr}^+(X)$ is contractible, we can unambiguously refer to the notion of that $\phi$ {\it preserves} or {\it reverses} orientation of $H^{+}(X)$. See for example \cite[Section 2]{KMT23first} for more details. Given an orientation-preserving diffeomorphism $f : X \to X$, we say that $f$ preserves (or reverses) orientation of $H^+(X)$ if the induced map $f^\ast : H^2(X;\R) \to H^2(X;\R)$ preserves (or reverses) orientation of $H^+(X)$.

Let us now recall the definition of a cork:
\begin{defi}
We say that a tuple $(Y,C,\tau)$ is a cork, if there exist contractible compact smooth 4-manifold $C$ with boundary $Y$ (which is always an integer homology sphere), and a diffeomorphism
\[
\tau: Y \rightarrow Y
\]
such that $\tau$ does not extend over $C$ as a diffeomorphism.
\end{defi}
Although, this definition is weaker than the one originally defined by Akbulut \cite{Ak91_cork}, we use it for simplicity. Sometimes we refer to a cork by $(C,\tau)$ omittinig the boundary. By the work of Freedman \cite{Fr82} $\tau$ does extend over $C$ as a homeomorphism. There are many examples in the literature of corks, see for example \cite{AY_plugs}. Following the work of \cite{MAt, CFHS} Corks plays a fundamental role in the study of exotic smooth structures. We call a cork to be of \textit{Mazur type}, if it is build by one 1-handle and 2-handle, which algebraically cancels the 1-handle. Sometimes we will also require the definition of a strong cork, first introduced by Lin-Ruberman-Saveliev \cite{LRS18}.
\begin{defi}
We call a pair $(Y,\tau)$ a {\it strong cork}, if $Y$ is an integral homology 3-sphere and $\tau : Y \to Y$ is an orientation-preserving diffeomorphism such that $Y$ bounds a contractible compact smooth 4-manifold, and that $\tau$ does not extend over \textit{any} integral homology 4-ball $W$ with $\partial W=Y$.
\end{defi} 
Many well-known example of corks also tend to be strong corks \cite{DHM22}. Such corks were detected by a variant of involutive Heegaard Floer homology. Building on the work of Baraglia \cite{Ba21}, in \cite[Theorem 1.7]{KMT23first}, the authors developed a tool stemming from the famliy Seiberg--Witten theory to detect strong corks. We state it below:
\begin{theo}\label{strong_cork}
Let $(Y , \phi)$ be an oriented homology 3-sphere with an orientation-preserving (not necessarily order 2) diffeomorphism $\phi$. Suppose that $Y$ bounds a 4-manifold $X$ with $b^+(X)=1$ and $b_1(X)=0$. 
Now if $\Phi$ is any orientation-preserving, smooth extension of $\phi$ so that $\Phi$ is $H^{+}$-reversing, and $\spinc$-preserving for some $\spinc$-structure $\s$ together with
\[
\frac{c_1(\fraks)^2-\sigma(X)}{8} > 0,
\]
then $(Y, \phi)$ is a strong cork. 
\end{theo}

This follows from the family analog of Donaldson's diagonaization theorem proven using Seiberg--Witten theory due to Baraglia~\cite{Ba21}.  We shall discuss a generalization of these results combined with real Seiberg--Witten theory in the next section.  The class of strong corks detected by \cref{strong_cork} is the main ingredient to provide equivalent exotic surfaces. 

\begin{defi}
    We call a cork $(C, \tau)$ is detected by \textit{spin$^c$ preserving family Seiberg--Witten theory} if the assumptions of \cref{strong_cork} are satisfied with certain tuple $(X, \Phi, \fraks)$. Let us call $(X, \Phi)$ (forgetting $\fraks$) a {\it filling as $S^1$-family}.  
\end{defi}

Now by means of an example we demonstrate how to detect corks using the obstruction from Theorem~\ref{strong_cork}. This was originally shown to be a strong cork by \cite{LRS18}, see also \cite{DHM22}.

\begin{ex}\label{ex: akbulut cork}
Consider the Akbulut cork $(C ,\tau)$. The boundary of this cork $Y$ can be identified with $(+1)$-surgery on the mirror of the $9_{46}$ knot, $Y:=S^3_{+1}(\bar{9}_{46})$. The associated cork-twist $\tau$ is depicted in Figure~\ref{fig: akbulut_example}. It was proved in \cite[Lemma 7.4]{DHM22}, that there is an equivariant negative-definite cobordism $W$ (with intersection form $2(-1)$) from $(Y,\tau)$ to $(\Sigma(2,3,7), \tau')$. That is the cobordism $W$ is equipped with a diffeomorphism $f$ which restricts to the diffeomorphism $\tau$ and $\tau'$ on $Y$ and $\Sigma(2,3,7)$ respectively, where $\tau'$ represents the involution of $\Sigma(2,3,7)$ obtained as the complex conjugation (thinking $\Sigma(2,3,7)$ of as a link of singularity).
\begin{figure}[h!]
\center
\includegraphics[scale=0.6]{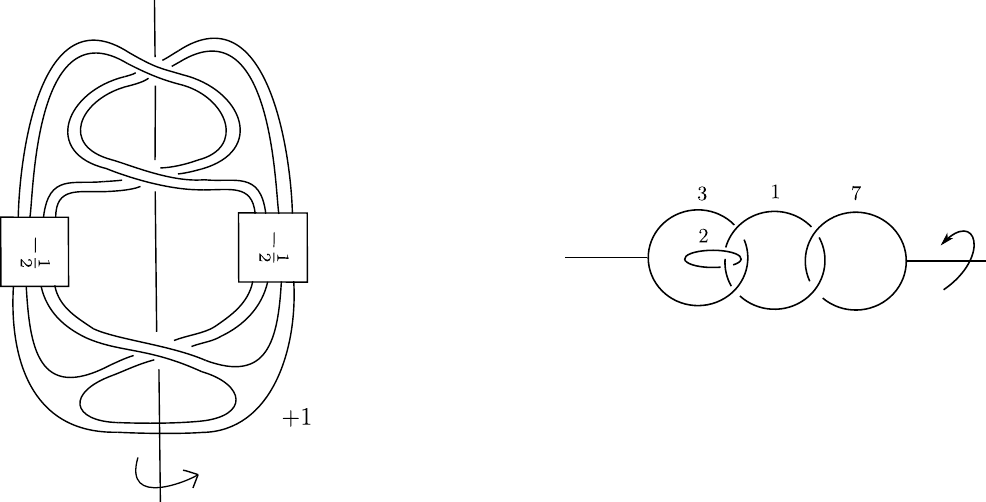}
\caption{Left: The boundary of the Akbulut cork. Right: A Kirby diagram of $(\Sigma(2,3,7), \tau')$  obtained as the boundary of the plumbing.}\label{fig: akbulut_example}
\end{figure}
Let us now define a 4-manifold $(W_0, f_{W_0})$ bounded by $(-\Sigma(2,3,7), \tau')$. To define the manifold $W_0$, recall that there is an embedding of $\Sigma(2,3,7)$ into the 4-manifold $K3$. $W_0$ is precisely defined as the complement of the Minor fiber $M(2)$ in $K3$ (note that Minor fiber has boundary $\Sigma(2,3,7)$). In particular, the intersection form of $W_0$ is given as $-E_8 \oplus H$, where $H$ represents the hyperbolic form. It can be seen from Saveliev \cite{saveliev1998dehn}, that $W_0$ admits a Kirby presentation given in Figure~\ref{fig: akbulut_example}. 
\begin{figure}[h!]
\center
\includegraphics[scale=0.6]{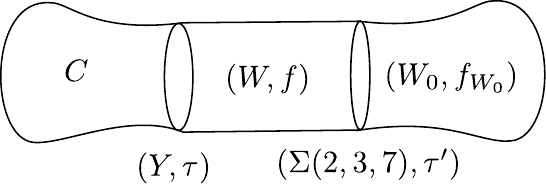}
\caption{Schematic of the $S^1$-family filling.}\label{fig: S1_family_filling}
\end{figure}
The involution $f_{W_0}$ on $W_0$ is then defined using strongly invertible symmetry of the diagram. Moreover in \cite[Lemma 4.6]{KMT23first}, it was shown that this extension of $f_{W_0}$ onto $W_0$ reverses the orientation of $H^{+}(W_0)$. Now we equip $W_0$ with its unique spin structure $\s_0$, moreover, by construction there is a $\spinc$-structure $\s_W$ on the negative-definite cobordism (given by the generators of $H_2$)  which is preserved by $f$. Let us define
\[
\s:= \s_{W} \cup \s_0, \; \; \mathrm{and} \; \; X := W \cup W_0.
\]
Then it follows that the diffeomorphism
\[
f_X:= f \cup f_{W_0}:  W \cup W_0 \rightarrow W \cup W_0.
\]
preserves the $\spinc$-structure $\s$ and reverses the orientation of $H^{+}(X)$, while satisfying 
\[
c_1(\fraks)^2-\sigma(X) > 0,
\]
Hence $(X,f_X)$ is a filling of $(Y,\tau)$ as a $S^1$-family and $(Y,\tau)$ is a strong cork, see Figure~\ref{fig: S1_family_filling}. Note that the intersection form of the closed 4-manifold $X \cup C$ is isomorphic to $(+1) \oplus 11(-1)$.
\end{ex}
This example will serve the cornerstone of our equivalent exotic surface detection results.

\section{A family version of real Seiberg--Witten theory}
\label{section: A family version of real Seiberg--Witten theory}

We develop a new tool to study exotic surfaces called a family version of real Seiberg--Witten diagonalization theorem over $S^1$, which employs a 1-parameter family version of real Seiberg--Witten theory. 
The original real Seiberg--Witten theory has been developed in \cite{TW09, Na13, Ka22, KMiT21, JLi22} over the years.

\subsection{Family real diagonalization theorem over $S^1$ }\label{Family real diagonalization theorem over $S^1$}
 Let $X$ be a simply-connected closed smooth oriented 4-manifold. 
Given an embedded surface $S$ in $X$, if we can take a branched covering of $X$ along $S$, we denote the cover by $\Sigma_2(S)$.
In the real Seiberg--Witten theory, we shall use certain kinds of surfaces embedded into 4-manifolds for taking the double-branched covering spaces.

\begin{defi}
A smoothly embedded not necessary connected (possibly un-orientable) surface $S$ in $X$ is called an {\it admissible surface} if $S$ satisfies the following conditions:
\begin{itemize}
   \item $[S] \equiv 0 \; (\operatorname{mod} 2)$, which ensures that there is the branched covering $\Sigma_2(S)$ of $X$ along $S$.
    \item The equality \begin{align}\label{b+=0}
       b^+(\Sigma_2(S)) = b^+(X)
    \end{align}
    is satisfied,
    \item There is a spin$^c$ structure $\fraks_R$ on  $\Sigma_2(S)$ such that 
    \begin{align}\label{add 3}
    \tau ^* \fraks_R \cong \overline{\fraks}_R
   \text{ and }
    c_1(\fraks_R)^2 - \sigma(\Sigma_2(S)) = 0, 
    \end{align}
    where $\tau$ is the branched covering involution on $\Sigma_2(S)$.
\end{itemize}

We call the pair $(X, S)$ an admissible pair.

\end{defi}

\begin{rem}
 Under the condition \eqref{b+=0}, from \cite[Theorem 1.1]{KMT24i}, one can see 
\[
  c_1(\fraks_R)^2  - \sigma (\Sigma_2(S))\leq 0.
\]
So the left-hand side of \eqref{add 3} cannot be positive. 
\end{rem}

For an admissible pair $(W, S)$, the conditions \eqref{b+=0} and \eqref{add 3} imply that the real Bauer--Furuta invariant of a pair $(X, S')$ of 4-manifold with the surface embedding and that of the pairwise connected sum $(W, S) \# (X, S')$ live in the same equivariant stable (co)homotopy group.  
Now it has been observed in \cite{Na00} that the existence of a spin structure on $\Sigma_2(S)$ is equivalent to 
\[
PD([S])/2 \equiv w_2(X) \mod 2
\]
and it is unique if it exists since we see $H^1(\Sigma(S); \Z_2)=0$ from $H_1(X; \Z_2)=0$, where $PD$ denotes the Poincar\'e duality. 

Now, we give concrete examples of admissible surfaces. 
\begin{lem}\label{ex adm sur} The following pairs are examples of admissible pairs. 
\begin{itemize}
\item $X=  S^4$, $S$ is any $2$-knot in $X$, 
\item $X=S^4$, $S$ is the standard $\R P^2$-knot in $X$ whose branched cover is negative-definite, 
    \item $X= \C P^2$, $S$ is a degree-$2$ complex curve in $\C P^2 =  X$, and
    \item $X =S^2 \times S^2$, $S= S^2 \times \{N, S\} \subset S^2 \times S^2$.
\end{itemize}

Also, a pairwise connected sum $S_1\# S_2 \subset X_1 \# X_2$ of two admissible surfaces $S_1 \subset X_1$ and $S_2 \subset X_2$ is also admissible. 
   
\end{lem}

\begin{proof}
    One can easily verify the conditions on $[S]$. The double branched cover of them are a $\Z_2$-homology $S^4$, $-\C P^2$, $S^2\times S^2$, and $S^2\times S^2$ respectively.  
    If the double-branched cover is a $\Z_2$-homology $S^4$ or $S^2 \times S^2$, one can use the unique spin structure on it as $\mathfrak{s}$. If the double-branched cover is  $-\C P^2$, one can use the generator of $H^2(-\C P^2; \Z) = \Z$ as the $\spinc$-structure $\s$. Then the condition~(\ref{add 3}) is satisfied. Since all conditions of an admissible surface are linear with respect to pairwise additions of homology classes and 4-manifolds, admissibility is preserved under the pairwise connected sum. This completes the proof. 
\end{proof}

Let $S$ be a smoothly embedded (possibly un-orientable) closed surface in a closed oriented 4-manifold $X$ with $[S] \equiv 0 \operatorname{mod} 2$ and with a spin$^c$ structure $\fraks_R$ on the double branched cover $\Sigma_2(S)$ along $S$ such that $
    \tau ^* \fraks_R \cong \overline{\fraks}_R$. 

    For such a choice of spin$^c$ structure $\fraks_R$, one can associate anti-complex linear involution $I$ on the spinor bundle compatible with the Clifford multiplication, which satisfies the commutative diagram: 
    \[
  \begin{CD}
     \mathbb
{S} @>{I}>> \mathbb
{S} \\
  @V{p}VV    @V{p}VV \\
     \Sigma(S)   @>{\tau}>>  \Sigma(S)
  \end{CD}
    \]
    where $\mathbb
{S}$ is the spinor bundle of $\fraks_R$ over $\Sigma_2(S)$ and $p : \mathbb{S} \to \Sigma(S)$ denotes the projection.  This is called a \textit{real structure}, and it was shown to exist in \cite[Lemma 2.10]{KMT24i}. 

The lift satisfies 
\[
I ( \rho (X) \cdot \phi(x) ) = \rho (d\tau (X))  \cdot I (\phi (x))
\]
for $x \in \Sigma_2 (S)$, $X \in T_x \Sigma_2 (S)$, $\rho$ is the Clifford multiplication and $\phi$ is a spinor. 
The involution $I$ also acts on $i\Lambda^1_{\Sigma_2(S)}$ by  $-\tau^*$. With respect to these actions, the Seiberg--Witten equations are $I$-equivariant.

In this situation, a key observation to get a 1-parameter family of real structures on mapping tori is the following.

\begin{prop}
\label{family I}

    Let $f: X\to X$ be an orientation-preserving diffeomorphism such that 
\[ f|_{S} = \id_S , 
       \text{ and } \wt{f}^*\fraks_R \cong \fraks_R,
       \]
where $\wt{f} : \Sigma_2 (S) \to  \Sigma_2 (S) $ is the diffeomorphism obtained as the lift of $f$. We further suppose $b_1(\Sigma_2(S)) =0$. 
Then, there is a spin$^c$ bundle map $\wt{\bf f}$ inducing a commutative diagram
\[
  \begin{CD}
     \mathbb{S} @>{\wt{\bf f}}>> \mathbb{S}  \\
  @V{p}VV    @V{p}VV \\
     \Sigma(S)   @>{\wt{f}}>>  \Sigma(S)
  \end{CD}
\]
which commutes with the real involution $I$ on $\mathbb{S}$,     
where $\mathbb
{S}$ is the spinor bundle of $\fraks_R$ over $\Sigma_2(S)$.
\end{prop}
\begin{proof}
In \cite[Lemma 2.7]{KMiT21}(See also \cite{JLi22}), the following uniqueness result is proven:  
If we have two real involutions $I$ and $I'$ of $S$ which cover the branched involutions, there is a gauge transformation $u : \Sigma_2(S) \to U(1) $ on $\Sigma_2(S)$ such that 
\[
 u I =  I' . 
\]

First, we take an arbitrary spin$^c$ lift of $\wt{f}$ to the spinor bundle $\wt{\bf f}':  \mathbb{S} \to  \mathbb{S} $. Then, for a given real structure $I$, we consider the real structure $\wt{\bf f}' \circ I \circ  (\wt{\bf f}')^{-1}$. By the uniqueness of real structures, we see there is a gauge transformation $u$ on $\Sigma_2(S)$ such that $
 u I =\sqrt{u} I  \sqrt{u}^{-1}  = \wt{\bf f}'\circ I  \circ (\wt{\bf f}')^{-1}$. Since $b_1(\Sigma_2(S))=0$, one can take a gauge transformation $\sqrt{u}: \Sigma_2(S) \to S^1$ such that $\sqrt{u}^2 =u$.
 (Precisely, if we have $u= \exp(i g(x))$, then $\sqrt{u} := \exp(i g(x)/2)$.)
Then, we get the desired lift as  
\[
\wt{\bf f} := \sqrt{u}^{-1} \circ \wt{\bf f}' : \mathbb{S}  \to \mathbb{S} . 
\]
Indeed, we have 
\[
 \sqrt{u} I  \sqrt{u}^{-1}  = u  I  = \wt{\bf f}'\circ I  \circ (\wt{\bf f}')^{-1}. 
\]
Thus 
\[
 I \circ  \sqrt{u}^{-1} \wt{\bf f}'  = \sqrt{u}^{-1} \wt{\bf f}'\circ I
\]
holds. 
This completes the proof. 
\end{proof}

Now, we state a 1-parameter family version of \cite[Theorem 1.1]{KMiT21} for double-branched covers using the existence of a family of real spin$^c$ structures verified in \cref{family I}.

\begin{theo}\label{thm: real S1}

Let $S$ be a smoothly embedded surface in a closed oriented 4-manifold $W$ whose $\Z_2$-homology class is zero.   Fix a spin$^c$ structure $\fraks_R$ on $\Sigma_2(S)$ such that $\tau^* \fraks_R \cong \overline{\fraks}_R$, where $\tau: \Sigma_2(S) \to \Sigma_2(S)$ is the covering involution. 
 Suppose $\dim H^+(\Sigma_2(S); \R)^{-\tau^*}=1$.  
We further suppose that there is an orientation-preserving diffeomorphism $f : W \to W$ such that 
 $f(S) = S$ and $f|_{S} = \id_S$, so that it induces a $\Z_2$-equivariant diffeomorphism 
\[
\wt{f} : \Sigma_2(S   ) \to \Sigma_2(S )
\]
with respect to the covering involution. 
Suppose that
\[
\wt{f}^* \fraks_R \cong {\fraks}_R,
\]
and that the diffeomorphism $\wt{f}$ acts on \[
H^+(\Sigma_2(S ); \R)^{-\tau^*} \cong \R 
\]
as an orientation-reversing automorphism of a vector space. 
Then we have 
\[
c_1(\fraks_R)^2-\sigma(\Sigma_2(S))\leq 0.
\]
\end{theo}

 \cref{thm: real S1} can be regarded as a non-free analog of a family and $Pin^-(2)$-version of the diagonalization theorem given in \cite{KN23}.

\begin{proof}[Proof of \cref{thm: real S1}]
From \cref{family I}, we can take a lift $\wt{\bf f} : \mathbb{S} \to \mathbb{S} $ of $\wt{f}$, which forms a family of real spin$^c$ structures over $S^1$: 
\begin{align}\label{family I1}
(\Sigma_2(S), \mathbb{S}, I) \to \mathbb{E} \to S^1
\end{align}
defined as the mapping torus of $(\wt{f}, \wt{\bf f})$ with fiber $(\Sigma_2(S), \mathbb{S}, I)$.
We can equip a fiberwise Riemannian metric with the underlying smooth fiber bundle with fiber $\Sigma_2(S)$ over the base $S^1$. 
First, we forget the real involution $I$. Then from \cite[Proposition 2.20]{BK19}, the family of spin$^c$ structures defines a family $S^1$-equivariant Bauer--Furuta invariant 
\[
{\bf BF}_{\Sigma_2(S), \fraks_R} : Th (E) \to Th (F), 
\]
where $E$ and $F$ are finite-rank vector bundles over $S^1$ having the following decompositions: 
\begin{itemize}
    \item $E$ and $F$ have decompositions $E \cong E(\R) \oplus E(\C )$ and $F \cong F(\R) \oplus F(\C )$ into real and complex vector bundles over $S^1$ and $Th$ denotes the Thom space, 
    \item $\dim F(\R) - \dim E(\R) = b^+(\Sigma_2(S))$,
    \item $\dim_\C E(\C) - \dim_\C F(\C) = \frac{1}{8} (c_1 (\fraks_R)^2 - \sigma(\Sigma_2(S)))$, 
    \item $S^1$ acts on $E(\R)$ and $F(\R)$ trivially and on $E(\C)$ and $F(\C)$ as fiberwise complex multiplication, 
    \item ${\bf BF}_{\Sigma_2(S), \fraks} $ is an $S^1$-equivariant based continuous map, 
    \item ${\bf BF}_{\Sigma_2(S), \fraks}^{S^1} $ comes from a linear inclusion $E(\R) \hookrightarrow F(\R)$, and
    \item the fiberwise cokernel of the inclusion $E(\R) \hookrightarrow F(\R)$ can be identified with the vector bundle 
    \[
    H^+(\mathbb{E}) = \bigcup_{x\in S^1} H^+ (\mathbb{E}_x; \R) \to S^1. 
    \]
\end{itemize}

From \cref{family I}, as the mapping torus with respect to $\wt{\bf f}$, we get a 1-parameter family of anti-complex linear $I$-actions on the spinor bundle $\mathbb{S}_b$ for each $b \in S^1$
\[
I_ b :\mathbb{S}_b \to \mathbb{S}_b ,
\]
where $\bigcup_{b \in S^1} \mathbb{S}_b \to S^1$ is the fiberwise spinor bundle obtained as the mapping torus of $(\Sigma(S), f, \fraks)$. 
Since the map ${\bf BF}_{\Sigma_2(S), \fraks} $ is obtained as a finite-dimensional approximation of the Seiberg--Witten equation which is $I_b$-equivariant on each $b \in S^1$, we see the map  ${\bf BF}_{\Sigma_2(S), \fraks} $ is $I = \{ I_b\}$-equivariant as well if we take an $I$-invariant finite-dimensional approximation. Note that the involution $I$ acts on the 1-form part by $-\tau^*$ and the spinor part anti-linearly.
Together with above $S^1$-action, we obtain the family $O(2)$-equivariant Bauer--Furuta invariant written by the same notation
\[
{\bf BF}_{\Sigma_2(S), \fraks} : Th (E) \to Th (F). 
\]
Now, we consider the $I$-invariant part : 
\[
{\bf BF}_{\Sigma_2(S), \fraks}^I  : Th (E^I) \to Th (F^I)
\]
again parametrized by $S^1$, with the following properties: 
\begin{itemize}
    \item $\dim F^I(\R) - \dim E^I(\R) = b^+(\Sigma_2(S))- b^+(X)$,
    \item $\dim_\R E^I(\C) - \dim_\R F^I(\C) = \frac{1}{8} (c_1 (\fraks_R)^2 - \sigma(\Sigma_2(S)))$, 
    \item $\Z_2 = \{ \pm 1\}\subset S^1$ acts on $E^I(\R)$ and $F^I(\R )$ trviailly and $E^I(\C)$ and $F^I(\C )$ as muplications as real vector bundles, 
    \item ${\bf BF}_{\Sigma_2(S), \fraks}^I$ is a $\Z_2= \{\pm 1\}$-equivaraint map,
    \item $({\bf BF}_{\Sigma_2(S), \fraks}^I)^{\{\pm 1\} }$ is induced from a linear inclusion $E(\R)^I \to F(\R)^I$, 
    \item the fiberwise cokernel of the inclusion $E(\R)^I \to F(\R)^I$ can be identified with the real vector bundle
    \[
    H^+_R(\mathbb{E}) := \bigcup_{x \in S^1} H^+ (\mathbb{E}_x; \R )^{-\tau_*}  = H^+(E)^{-\tau^*} \to S^1. 
    \]
\end{itemize}

Let us denote the $I$-invariant part map by 
\[
{\bf BF}_{R}(S,  \fraks, f) :=  {\bf BF}_{\Sigma_2(S), \fraks}^I  : Th (E^I) \to Th (F^I)
\]
and call it the {\it 1-parameter family real Bauer--Furuta invariant}. 
Now, we are ready to prove the result. Consider the following diagram: 
\[
  \begin{CD}
       Th (E^I)    @>{{\bf BF}_{R}(S,  \fraks_R, f)}>> Th (F^I) \\
  @A{i_0}AA    @A{i_1}AA \\
     Th (E^I)^{\Z_2}     @>{{\bf BF}_{R}(S,  \fraks_R, f)^{\Z_2}}>>  Th (F^I)^{\Z_2}  
  \end{CD}
\]
and the induced map on $\Z_2$-equivariant cohomologies: 
\[
  \begin{CD}
       H^*_{\Z_2} (Th (E^I) )    @<{{{\bf BF}_{R}(S,  \fraks_R, f)  }^*}<< H^*_{\Z_2} (Th (F^I)) \\
  @V{i_0^*}VV    @V{i_1^*}VV \\
     H^*_{\Z_2} (Th (E^I)^{\Z_2})      @<<< H^*_{\Z_2} ( Th (F^I)^{\Z_2} ) .
  \end{CD}
\]
Now, we are assuming that $\wt{f}$ acts on $H^+ (\Sigma_2(S); \R )^{-\tau_*} \cong \R$ by $-1$. Thus, we see 
\[
0 \neq w_1 (H^+_R (\mathbb{E})) \in H^1(S^1; \Z_2). 
\]
Then we repeat a cohomological discussion given in \cite[Proof of Theorem 3.1]{KN23} (modelled on \cite{Ba21}, using $S^1$-equivariant cohomology in place of $\Z_2$-equivariant cohomology) to obtain 
\[
c_1^2 (\fraks_R) - \sigma(\Sigma_2(S))\leq 0. 
\]
This completes the proof. 
\end{proof}

\begin{rem}
     As the usual proof of the invariance of the Bauer--Furuta invariant, we expect that (without any added difficulty) the 1-parameter family version of real Bauer--Furuta invariant ${\bf BF}_{R}(S,X,  \fraks, f)$ gives a well-defined invariant of a tuple $( X, S, \fraks, f) $, by passing to a certain stable homotopy setup: 
     \begin{itemize}
         \item a smoothly embedded surface $S$ in $X$,
         \item a real spin$^c$ structure $\fraks$ on the double covering space, and 
        \item  an orientation-preserving diffeomorphism $f : X \to X$ such that $f(S)=S, f|_{S} = \id_S$ and $\wt{f}^* \fraks \cong \fraks$, where $\wt{f}$ is the diffeomorphism on the cover induced from $f$.
     \end{itemize}

\end{rem}

Now, we apply \cref{thm: real S1} to obtain the following obstruction to smooth isotopies: 
\begin{theo}\label{main real}
Let $W$ be a closed oriented smooth 4-manifold and $S$ be an admissible surface in $X$.
 Let $(Y, f)$ be a strong cork detected by spin$^c$ preserving family Seiberg--Witten theory. 
Suppose there is an extension $\wt{f}$ of $f$ to  $C \# W$ as an orientation-preserving diffeomorphism such that $\wt{f}^*= \id$ on $H^* (C \# W; \Z)$ for the contractible bound $C$ of $Y$.  
Then, $\wt{f} (S)$ and $S$ in $C \# W$ are not smoothly isotopic as embeddings rel boundary. 

More strongly, if we take a filling $(Z, f_Z)$ of $(Y, f)$ as $S^1$-family, 
\[
S  \; \text{and} \; (f_Z \# \wt{f}) (S) \subset  Z \cup_Y C \# W
\]
are not smoothly isotopic as embeddings. 

\end{theo}
\begin{proof}[Proof of \cref{main real}]
Note that the second statement for the closed 4-manifold $ Z \cup_Y C \# W$ implies the first statement. Thus, we prove the second. 
From the assumption, we have a filling $(Z, f_Z)$ of the cork $C$ by 
\[
Z \cup C 
\]
with an extension $f_Z$ of the cork action $f$ on $Z$ which preserves a spin$^c$ structure $\fraks_Z$ on $Z$ and reverses the orientation of $H^+(Z; \R) \cong \R$ such that 
\[
c_1^2 (\fraks_Z) - \sigma (Z) >0. 
\]
Now, we consider a pair 
\[
 X := Z \cup_Y C \# W, \quad S \subset  W. 
\]
From the assumption, the diffeomorphism $f$ also extends to $C \# W$ as well, which we denote by $\tilde{f}$. By abusing notation, let us still use $\tilde{f}$ for the diffeomorphism  of $X$ obtained by gluing $f_Z$ and $\tilde{f}$
\[
f_Z \cup \tilde{f} : X \to X. 
\]
Suppose $\wt{f}(S)$ and $S$ are smoothly isotopic in $C \# W$. Then, by composing an ambient isotopy, we can assume that $\wt{f}$ is $\id$ near the neighborhood of $S$ in $X$.
Therefore, we can induce a $\Z_2$-equivariant diffeomorphism with respect to the covering transformation. 
\[
\wt{f}^c : \Sigma_2(S; X) \to \Sigma_2(S; X)
\]
which covers $\wt{f} : X \to X$, where $\Sigma_2(S; X)$ denotes the branched covering space of $X$ along $S$.
Let us denote the covering involution of $\Sigma_2(S; X)$ by $\tau$.
Note that we have the following decomposition: 
\begin{align}\label{decomp}
\Sigma_2(S; X) = \left(Z \cup C\right)    \# \Sigma_2 (S ; W) \#  \left(Z \cup C\right). 
\end{align}
Note that the induced action of the diffeomorphism $\wt{f}^c$ on homology preserves the components $H_*(Z \cup C) \oplus H_*(Z \cup C)$ and $\tau$ flips the components $H_*(Z \cup C) \oplus H_*(Z \cup C)$.

We consider the following spin$^c$ structure $\fraks$ on $\Sigma_2(S; X)$
\begin{align}\label{decomp of spinc}
 \fraks := \fraks_Z \# \fraks_R \# \overline{{\fraks}}_Z, 
\end{align}
where $\fraks_R$ is a spin$^c$ structure on $\Sigma_2 (S ; W)$ obtained from the definition of admissible surface.  Then, we check that the following conditions are satisfied: 
\begin{itemize}
    \item[(i)] $(\wt{f}^c )^*  \fraks \cong  \fraks$ and $\tau^* \fraks \cong \overline{\fraks}$,  
    \item[(ii)] $(H^ + (\Sigma_2(S; X) ) )^{- \tau^*}\cong \R$, 
    \item[(iii)] $(\wt{f}^c )^*$ acts on $(H^ + (\Sigma_2(S; X) ) )^{- \tau^*}\cong \R$ as an orientation-reversing isomorphism of a vector space. 
\end{itemize}

The claim (i) follows from the following discussion. We have
\[
(\wt{f}^c)^* \fraks =(\wt{f}^c)^* (\fraks_Z \# \fraks_R \# \overline{{\fraks}}_Z ) =  (\wt{f}^c)^*\fraks_Z \# (\wt{f}^c)^* \fraks_R \# (\wt{f}^c)^* \overline{{\fraks}}_Z  \cong \fraks_Z \# \fraks_R \# \overline{{\fraks}}_Z,
\]
since $(\wt{f}^c)^*\fraks_Z \cong \fraks_Z$ from the construction and $(\wt{f}^c)^* \fraks_R \cong \fraks_R$ from the assumption so that $f^*$ is $\id$ on $C \# W$. Similarly, we have 
\[
\tau^* \fraks = \tau^* (\fraks_Z \# \fraks_R \# \overline{{\fraks}}_Z ) \cong \overline{{\fraks}}_Z \# \tau^* \fraks_R\# \fraks_Z \cong \overline{\fraks}
\]
since $\tau^* \fraks_R \cong \overline{\fraks}_R$. 
Next, we see (ii). Again from the description \eqref{decomp}, 
\[
b^+ ( \Sigma_2(S; X)) = b^+(\Sigma_2(S; W)) + 2 b^+(Z)= b^+(\Sigma_2(S; W)) + 2. 
\]
From the construction of $X$, one has
\[
b^+ (X) = b^+(W) + 1. 
\]
Therefore, from the assumption \eqref{b+=0} of admissible surface $S$, we have 
\[
\dim (H^ + (\Sigma_2(S; X) ) )^{- \tau^*}) = b^+(\Sigma_2(S; W)) + 2 - ( b^+(W) + 1) = 1 . 
\]
Now, we consider (iii), which describes the action of $\wt{f}^c$ on $H^ + (\Sigma_2(S; X) ) )^{- \tau^*}$. From (ii), $H^ + (\Sigma_2(S; X) ) )^{- \tau^*}$ is one-dimensional and it is represented by the anti-diagonal subspace 
\[
\R \cong \{ (h, h')  \in  H^+(Z ) \oplus H^+(Z)| h= -h'\}  \subset H^ + (\Sigma_2(S; X) ) 
\]
since $\tau$ is the flipping action on $H^+(Z ) \oplus H^+(Z)$. From the construction of $\wt{f}^c$, we see $(\wt{f}^c)^*$ preserves each component of $H^+(Z ) \oplus H^+(Z)$ and acts as multiplication by $-1$.
This proves the desired condition (iii). Now, we apply \cref{thm: real S1} to the data $(X, S, \wt{f}^c) : (\Sigma_2(S),\s) \to (\Sigma_2(S), \fraks)$ to obtain 
\[
c_1(\fraks)^2-\sigma(\Sigma_2(S; X))\leq 0. 
\]
Again from the decomposition \eqref{decomp}, the signature of $\Sigma_2(S; X)$ can be computed as 
\[
\sigma(\Sigma_2(S; X))= 2\sigma(Z) + \sigma ( \Sigma_2(S; W)). 
\]
From \eqref{decomp of spinc}, one can compute 
\[
c_1(\fraks)^2 = c_1(\fraks_Z)^2 + c_1(\fraks_R)^2 + c_1(\overline{\fraks}_Z)^2 = c_1(\fraks_R)^2 + 2 c_1(\fraks_Z)^2. 
\]
Thus, we get 
\begin{align*}
    c_1(\fraks)^2-\sigma(\Sigma_2(S; X)) & = c_1(\fraks_R)^2 + 2 c_1(\fraks_Z)^2 -2\sigma(Z) - \sigma ( \Sigma_2(S; W)) \\
    & = c_1(\fraks_R)^2 - \sigma ( \Sigma_2(S; W)) + 2 (c_1(\fraks_Z)^2 -\sigma(Z))\\
    & = 2 (c_1(\fraks_Z)^2 -\sigma(Z)) >0 
\end{align*}
which is a contradiction. 
Therefore, $\wt{f}(S)$ and $S$ are not smoothly isotopic. This completes the proof. 
\end{proof}

\section{Diffeomorphisms out of corks} \label{section: Diffeomorphisms out of corks}
The goal of this section is to establish that simple effective corks (Definition~\ref{defi: weakly effective cork}) exist. This is a somewhat standard result, although not well-documented in the literature. We will later use the simple corks to construct a diffoemorphism of a suitable closed 4-manifold.
We start with the following:
\begin{prop}\label{prop: mazur_cork}
$(Y,C,\tau)$ be a Mazur type cork. Moreover, suppose that $(Y,C,\tau)$ does not survive stabilization by $S^2 \times S^2$ (or $\mathbb{C}P^2$). Then $(Y,C, \tau)$ is a simple effective cork that is also SW-effective.
\end{prop}

\begin{proof}
This follows immediately from the work of Akbulut-Yasui \cite{AY_stein_corks}, who showed that any Mazur type corks admits an effective embedding in a symplectic closed manifold $X$. The SW-effective part of the claim is also shown in \cite{AY_stein_corks}.
\end{proof}

We now give infinitely many examples of corks satisfying the hypothesis of Proposition~\ref{prop: mazur_cork}.
\begin{prop}\label{prop:weak_cork_example}
For any integer $n$, there exist contractible compact smooth 4-manifolds $C_n$ with the prescribed boundary involution $\tau$ that have the following properties:
\begin{enumerate}
\item They are strong corks for any $n$.

\item $(\partial C_n,C_n,\tau)$ is a simple effective (and SW-effective) cork, for stabilization by either $S^2 \times S^2$ or $\mathbb{C}P^2$.
\end{enumerate}
\end{prop} 
\begin{proof}
There are several approaches to this. We use techniques from Auckly-Kim-Ruberman-Melvin \cite{AKHMR15}. Let $C_n$ be defined as the Mazur type corks shown in Figure~\ref{fig: family_of_corks} with the symmetry $\tau$ (Note that $(C_0, \tau)$ is in fact the Akbulut cork). First, we observe that $C_n$ does not survive stabilization by $\mathbb{C}P^2$. This is essentially shown in \cite[Theorem 4.2]{AKHMR15}. We include a brief discussion for readers' convenience. In the proof of \cite[Theorem 4.2]{AKHMR15}, it is shown that $C_n \# \mathbb{C}P^2$ is diffeomorphic to the trace of the $(+1)$-surgery along slice knots $K_n$, $X_{+1}(K_n)$, see Figure~\ref{fig: family_of_corks}. The knots $K_n$ are strongly invertible with the strong involution $\tau$ shown in Figure~\ref{fig: family_of_corks}.
\begin{figure}[h!]
\center
\includegraphics[scale=1.0]{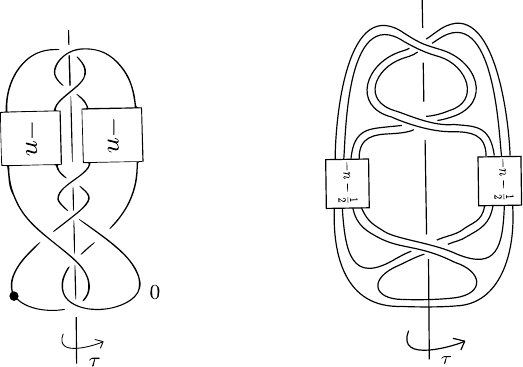}
\caption{\raggedright Left: The corks $(C_n, \tau))$. Right: The knots $K_n$ with the involution $\tau$. The box with label $(-n -1/2)$ represents $(-2n -1)$-many negative half-twists.}\label{fig: family_of_corks}
\end{figure}
 Moreover, the aforementioned identification between  $C_n \# \mathbb{C}P^2$ and $X_{+1}(K_n)$ is equivariant with respect to the involutions on the boundary of $\partial C_n$ and $S^3_{+1}(K_n)$. Since the boundary involution on $S^3_{+1}(K_n)$ extends over  $X_{+1}(K_n)$ (see for example \cite[Section 5]{DHM22}), it follows that $\tau$ extends over $C_n \# \mathbb{C}P^2$. 

The argument for extending $\tau$ on $C_n \# S^2 \times S^2$ is not explicit in the above argument but follows from the techniques developed therein. It was shown in the proof of \cite[Theorem 4.2]{AKHMR15} that the sphere $\mathbb{C}P^1 \subset \mathbb{C}P^2$ corresponds to the sphere $S_n$ obtained by capping off a particular slice disk $D_n$ of $K_n$ with the core of the 2-handle in $X_{+1}(K_n)$, under the identification discussed above. Now by capping off the core of the 2-handle in $X_{+1}(K)$ with the reflected copy of the disks $D_n$, we get another set of spheres $S^{\tau}_n$. It follows that if blowing down $S_n$ in $X_{+1}(K_n)$ results into $C_n$, while blowing down $S^{\tau}_{n}$ results into $C_n$ with the $\partial C_n$ identified by a twist by $\tau$. Hence, in order to show that $\tau$ extends over $C_n \# S^2 \times S^2$, it is enough to show that $S_n$ and $S^{\tau}_{n}$ are isotopic in $X_{+1}(K_n) \# S^2 \times S^2$. To achieve show this we use \textit{stable isotopy} from the proof of \cite[Theorem 4.1]{AKHMR15}. By following the stable isotopy shown in \cite[Section 2]{AKHMR15}, we get that the $S_n$ and $S^{\tau}_n$ are isotopic.
\begin{figure}[h!]
\center
\includegraphics[scale=0.9]{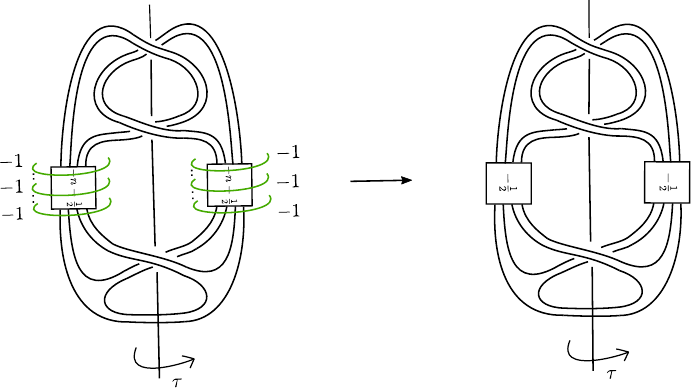}
\caption{The cobordism used in the Proof of Theorem~\ref{prop:weak_cork_example}, similar cobordism was also used in \cite{alfieri2023involutions}.}\label{fig: cork_family_coordism}
\end{figure}

The fact that $(\partial C_n, \tau)$ are strong corks, this is immediate from the arguments from \cite{DHM22}. As shown in Figure~\ref{fig: cork_family_coordism}, there are equivariant negative-definite corbordism from $(\partial C_n, \tau)$ to $(C_1,\tau)$. The invariant defined in \cite[Theorem 1.1]{DHM22} for $(\partial C_n,\tau)$ is strictly negative \cite[Theorem 1.13]{DHM22}. By construction $(C_n,\tau)$ are of Mazur type, hence Proposition~\ref{prop: mazur_cork} completes the proof.
\end{proof}

\section{Proofs of main theorems}
\label{section: Proofs of main theorems}

\subsection{Exotic strongly equivalent surfaces in once stabilized corks}We are now in place to provide proofs of our main theorems. We begin by proving Theorem~\ref{theo: intro sphere from adjunction}, which is broken into a couple of Lemma. We will also need the families adjunction inequality proved by Baraglia \cite{Ba20}.
Let us briefly recall the {\it family Seiberg--Witten invariant} (see such as \cite{Ru98,LL01,BK18} for details): Let $(X,\fraks)$ be a closed smooth spin$^c$ 4-manifold with $b^+(X)\geq3$.
Suppose that $d(\fraks)$, the Seiberg--Witten formal dimension, is $-1$.
Then, for an orientation-preserving diffeomorphism $f : X \to X$ that satisfies $f^\ast\fraks \cong \fraks$, we can define a $\Z/2$-valued invariant $\FSW(X,\fraks,f)$ called the family Seiberg--Witten invariant.

\begin{theo}[{\cite[Theorem 1.2]{Ba20}}]\label{thm: family adjunction}
Let $X$ be an oriented smooth closed 4-manifold with $b^+(X)\geq3$.
Let $\fraks$ be a spin$^c$ structure on $X$ with $d(\fraks)=-1$, and let $f : X \to X$ be a diffeomorphism that preserves orientation and $\fraks$, and assume that $\FSW(X,\fraks,f)\neq0$.
Let $i : \Sigma \hookrightarrow X$ be a closed connected oriented smoothly embedded surface in $X$ with $[\Sigma]^2\geq0$.
Suppose that the embeddings $i: \Sigma \hookrightarrow X$ and $f\circ i : \Sigma \hookrightarrow X$ are smoothly isotopic to each other.
\begin{enumerate}
\item If $g(\Sigma)>0$, we have 
\[
2g(\Sigma)-2 \geq |c_1(\fraks)\cdot [\Sigma]|+[\Sigma]^2.
\]
\item If $g(\Sigma)=0$, $[S]$ is a torsion class.
\end{enumerate}

\end{theo}

Firstly, we prove the following Lemma. 
\begin{lem}
\label{lem: sphere from adjunction}
Let $C$ be a simple SW-effective cork and let $S$ be a smoothly embedded 2-sphere in $C\#S^2\times S^2$ whose homology class satisfies $[S]^2 \geq 0$ and $[S]\neq0$.
Then there exist infinitely many smoothly embedded 2-spheres $\{S_i\}_{i=1}^\infty$ in $C\#S^2\times S^2$ that are strongly equivalent to $S$ but relatively exotic to each other.
\end{lem}

\begin{proof}[Proof of Lemma~\ref{lem: sphere from adjunction}]
Let $C$ be a simple effective cork, with the underlying embedding in the closed 4-manifold $X$. Moreover, we connected sum a copy of $S^2 \times S^2$ such that
\[
C \# S^2 \times S^2 \hookrightarrow X \# S^2 \times S^2.
\]
Let $S$ be any smoothly embedded sphere in $C \# S^2 \times S^2$ such that $[S] \in H_{2}(S^2 \times S^2; \mathbb{Z})$ is a non-zero class with $[S]^{2} \geq 0$. 

First, note that definition of a simple SW-effective cork implies that $\tau$ on $\partial C$ extends over $C \# S^2 \times S^2$, there is an induced diffeomorphism
\[
f: X \# S^2 \times S^2 \rightarrow X^{\tau} \# S^2 \times S^2.
\]
where $X^{\tau}$ is given by regluing $C$ by the boundary twist, $X^\tau:= X \setminus C \cup_{\tau} C$. Note that $f$ maps $C \# S^2 \times S^2$ to itself. Moreover, the it also follows from the definition that there is a $\spinc$ structure $\fraks$ of $X$ such that 
\[
\SW(X, \fraks) \neq \SW(X^{\tau}, \fraks) \; \mathrm{mod} \; 2.
\]
We will use these two facts to construct a diffeomorphism
\[
\Phi: X \# S^2 \times S^2 \rightarrow X \# S^2 \times S^2,
\]
as a composition of diffeomorphisms defined as follows: We now consider the diffeomorphism $\phi$ of $S^2 \times S^2$ defined by the product of the complex conjugations 
\begin{align}\label{compconj}
([z_1; z_2], [z'_1; z'_2]) \mapsto ([\overline{z}_1; \overline{z}_2], [\overline{z}'_1; \overline{z}'_2]):  \C P^1\times \C P^1 \to \C P^1\times \C P^1.
\end{align}
We extend $\phi$ to a diffeomorphism of $X \# S^2 \times S^2$ by identity (after isotoping $\phi$ on a 4-disk in $S^2 \times S^2$, so that $\phi$ fixes the disk pointwise). Similarly, we also extend $\phi$ to $X^{\tau} \# S^2 \times S^2$. We continue to call both of these diffeomorphims as $\phi$, by abusing notation. Then $\Phi$ is defined as follows:
\[
\Phi:= \phi \circ  f^{-1} \circ \phi \circ f.
\]
Importantly, note that $\Phi$ maps $C \# S^2 \times S^2$ to itself. 

Let $\frakt$ be the spin structure on $S^2\times S^2$ and set $\fraks' = \fraks\#\frakt$, which is a spin$^c$ structure on $X\#S^2\times S^2$.
Now following a standard argument in family Seiberg--Witten theory, having roots in Ruberman's work \cite{Ru98}, we show that
\[
\mathrm{FSW}(X \# S^2 \times S^2, \mathfrak{s}', \Phi) \neq 0.
\]
Indeed, we have
\begin{align*}
\FSW(X \# S^2 \times S^2,\fraks,\Phi)
&= \FSW(X \# S^2 \times S^2,\fraks,f^{-1} \circ \phi \circ f)
+ \FSW(X \# S^2 \times S^2,\fraks,\phi)\\
&= \FSW(X^\tau \#S^2\times S^2,\fraks'\#\frakt, \phi)
+ \FSW(X \# S^2 \times S^2,\fraks,\phi)\\
&= \SW(X^\tau,\fraks')
+ \SW(X,\fraks')
= 1+0=1
\end{align*}
in $\Z/2$, where in the first line we have used \cite[Lemma 2.6]{Ru98} for the additivity of the family Seiberg--Witten invariant, in the second line we have used the conjugation invariance and in the third line we used \cite[Theorem 9.5]{BK18}. Now we claim that $S$ and $\Phi^{k}(S)$ satisfy the following conditions (for any integer $k$):

\begin{enumerate}
 \setlength\itemsep{1em}
\item $S$ and $\Phi^{k}(S)$ are topologically isotopic relative in $C \# S^2 \times S^2$, relative to the boundary.

\item There is a diffeomorohism 
\[
g_k: C \# S^2 \times S^2 \rightarrow C \# S^2 \times S^2
\]
that sends $S$ to $\Phi^{k}(S)$. Moreover, $g_k|_{\partial}= \mathrm{id}$ and 
\[
(g_{k})_{*}=\mathrm{id}: H_2(C \# S^2 \times S^2, \partial) \rightarrow H_2(C \# S^2 \times S^2, \partial).
\]
\end{enumerate}

Indeed, if the aforementioned conditions are met, it follows from Definition~\ref{defi:strong_weak_equi_boundary} that $S$ and $\Phi^{k}(S)$ are strongly equivalent surfaces in $C \# S^2 \times S^2$, for positive integer $k$. Indeed, this follows immediately by  taking $g_k:= \Phi^{k}$. As discussed before $\Phi^{k}$ on $X \# S^2 \times S^2$ induces a map that is supported in $C \# S^2 \times S^2$. Hence, $\Phi$ is a diffeomorphism of $C \# S^2 \times S^2$. Moreover, $\Phi^{k}$ acts an identity on the $H_2(C \# S^2 \times S^2, \partial)$ by definition, and hence it follows from \cite{OP22} that $S$ and $\Phi^{k}(S)$ are topologically isotopic relative to the boundary. It remains to check that $S$ and $\Phi^{k}(S)$ are not smoothly isotopic to each other (relative to the boundary). 

Towards contradiction, suppose that they are. 
By the definition of the SW-effectiveness, $X$ is simply-connected hence $[S]\neq0$ is not a torsion class. This is a contradiction by Theorem~\ref{thm: family adjunction}, (2).
This gives the required infinite family of strongly equivalent surfaces $\{ S_i \}$ that are relatively exotic to each other. 

It follows from Proposition~\ref{prop:weak_cork_example} that there are infinitely many corks $C$ that admit an simple SW-effective embedding.
 \end{proof} 
 
We now move on to proving the higher genus version of Lemma~\ref{lem: sphere from adjunction}.

\begin{lem}\label{lem: surfaces from adjunction}
Let $C$ be a simple SW-effective cork and let $S$ be a oriented closed surface smoothly embedded in $C\#S^2\times S^2$.
If $g(S)>0$ and $2g(S)-2<[S]^2$, then there exist infinitely many smoothly embedded surfaces $\{S_i\}_{i=1}^\infty$ in $C\#S^2\times S^2$ that are strongly equivalent to $S$ but relatively exotic to each other.
\end{lem}

\begin{proof}[Proof of Lemma~\ref{lem: surfaces from adjunction}]
The proof is similar to that of Lemma~\ref{lem: sphere from adjunction}. In this case, we get the contradiction from the adjunction inequality
\[
2g(S)-2 \geq |c_1(\fraks) \cdot [S]| + S \cdot S > 2g(S) - 2
\]
and Theorem~\ref{thm: family adjunction}, (1).
This completes the proof.

\end{proof}

We are now in place to prove Theorem~\ref{theo: intro sphere from adjunction}.

\begin{proof}[Proof of Theorem~\ref{theo: intro sphere from adjunction}]
Let us write $\alpha = p[S^2 \times \{\ast\}] + q [\{\ast\} \times S^2]$ with $p,q \in \Z$.
Let $S \subset S^2\times S^2$ be a genus-minimizing oriented surface in $\alpha$.
Since the $g(S)=0$ case follows from \cref{lem: sphere from adjunction}, suppose that $g(S)>0$.
Then we have $|p|\geq 1$ and $|q|\geq1$ by \cite[Theorem 1]{Kuga84}, and it follows from \cite[Theorem 1]{LiLi98} that $g(S) = (|p|-1)(|q|-1)$.
This implies $2g(S)-2 < 2pq =[S]^2$, and thus the assertion follows from \cref{lem: surfaces from adjunction}.
\end{proof}

\begin{proof}[Proof of Theorem~\ref{thm: internal_stab_intro}] It was shown in \cite{Ak91_cork}, that there is an embedding of the Akbulut cork $W$ in $E(2) \# \overline{\mathbb{C}P}^2$. For the ease of notation, we will denote $E(2) \# \overline{\mathbb{C}P}^2$ by $X$. There is also an embedding of the Gompf nucleus $N(2)$ in $X$ which is away from $W$. Now from \cite{AY_knotted_corks} we get that regluing $W$ by the boundary twist in $X$ splits off a $S^2 \times S^2$ summand. Let us denote the the cork-twisted $X$ by $X^{\tau}$. Now it follows from \cite{akbulut2002variations, auckly2003families, baykur2013round} that there is a diffeomorphism
\[
\phi_{p}: X^{\tau} \rightarrow X^{\tau}_{2p+1},
\]
where here $X^{\tau}_{2p+1}$ represents 4-manifold obtained from $(2p+1)$-log transform on the specified nucleus inside $X^{\tau}$ for $ p\geq 0$. Let us now define
\[
W_p:= \phi_p^{-1}(W).
\] 
For our construction, We will also need two more diffeomorphisms. First, note that the diffeomorphism $\tau$ of the boundary of $W$, extends over $W \# S^2 \times S^2$. Hence this induces a diffeomorphism
\[
f_1: X \# S^2 \times S^2 \rightarrow X^{\tau} \# S^2 \times S^2.
\] 
Similarly for the embedded copy of twisted $W$ in $X^{\tau}_{2p+1}$, the boundary diffeomorphism extends over $W \# S^2 \times S^2$. Moreover, regluing the twisted $W$, changes $X^{\tau}_{2p+1}$ to $X_{2p+1}$. Hence, as before there is an induced diffeomorphism
\[
f_2 : X^{\tau}_{2p+1} \# S^2 \times S^2 \rightarrow X_{2p+1} \# S^2 \times S^2.
\] 
Here, $X_{2p+1}$ represents the 4-manifold obtained from $X$ by $(2p+1)$-log transforming the distinguished nucleus in $X$. Using these above diffeomorphisms, we will firstly define a diffeomorphism $\phi'_p$. It is defined as a composition  
\[
\phi'_{p}:=f_2 \circ \phi_{p} \circ f_1: X \# S^2 \times S^2 \rightarrow X_{2p+1} \# S^2 \times S^2.
\]
Here by abusing notation we have denoted $\phi_p \# \mathrm{id}: X^{\tau} \# S^2 \times S^2 \rightarrow X^{\tau}_{2p+1} \# S^2 \times S^2$ as $\phi_p$.

We now argue that there is a contractible 4-manifold (with boundary) $C \subset X$ such that 
\[
\phi'^{-1}_{p}(S^2 \times S^2) \subset C \# S^2 \times S^2.
\]
Note that, as discussed in the beginning, there is an embedding of the Akbulut cork $W$ in $X \# S^2 \times S^2$. Moreover, we can consider another embedding of the Akbulut cork $W'_p \subset X \# S^2 \times S^2$, which is given by
\[
W'_{p}:= f^{-1}_1(W_p).
\]
We now claim that after moving both $W$ and $W'_{p}$ by isotopy, we can find a contractible 4-manifold $C \subset X \# S^2 \times S^2$ such that $C$ contains a disjoint embedding of $W$ and $W'_{p}$. To achieve this, we follow the arguments from \cite{MS_higher_corks}. To see this, first observe that both $W$ and $W'_{p}$ are AC-corks, in the language of \cite[Definition 1.1]{MS_higher_corks}. Indeed, they are both Mazur type cork with one 1-handle and one 2-handle. Then, by using an argument as in the Finger Lemma from \cite{MS_higher_corks}, we obtain that there is an embedding of a AC-manifold of type-I \cite[Definition 1.1]{MS_higher_corks}, containing $W \cup W'_{p}$. Then by the encasement lemma from \cite{MS_higher_corks}, there is another AC-manifold $C$ of type-II  $H_2(C)= H_2(W \cup W'_{p})=0$. In particular, $C$ is a contractible manifold, which by construction contains an embedding of $W$ and $W'_{p}$. We claim that for this $C$, we have
\[
\phi'^{-1}_{p}(S^2 \times S^2) \subset C \# S^2 \times S^2.
\]
Indeed, a point $x \in S^2 \times S^2 \subset X_{2p+1} \# S^2 \times S^2$ gets mapped to $W^{\tau} \# S^2 \times S^2 \subset X^{\tau}_{2p+1} \# S^2 \times S^2$ under $f^{-1}_2$, which is then sent to a point in $W_p \# S^2 \times S^2 \subset X^{\tau} \# S^2 \times S^2$. Finally, $f^{-1}_1$ maps that point to a point in $C \# S^2 \times S^2$.

We will use this diffeomorphism $\phi_{p}$ to construct an exotic diffeomorphism of $X \# S^2 \times S^2$. This is done by considering the commutator of $\phi_{p}$ and the reflection map of $S^2 \times S^2$ (which we will call $g$) defined in \eqref{compconj}. That is we define
\[
\Phi_p:= [\phi'_{p},g]: X \# S^2 \times S^2 \rightarrow X \# S^2 \times S^2.
\]
Note that by construction $\Phi_p$ supported on $C \# S^2 \times S^2$. Now let $S \subset S^2 \times S^2 \subset X \# S^2 \times S^2$ be a surface with $S^2 \geq 0$. We claim that $\Phi_{p}(S)$ and $S$ are pairs of strongly equivalent exotic surface in $C \# S^2 \times S^2 \subset X \# S^2 \times S^2$. Moreover, for suitably chosen $p$, $\Phi_{p}(S)$ and $S$  remain strongly equivalent exotic pairs after $n$-many internal stabilization by $T^2$. This part of the argument is motivated from \cite{Au23}.

First, let us fix a basis of the intersection form of $X \# S^2 \times S^2$ which is isomorphic to $2(+1) \oplus 19 (-1) \oplus 2 H$, by $\{ S_i \}^{2}_{i=1} \oplus \{ R \}^{19}_{i=1} \oplus \{ P_0, Q_0, P_1, Q_1 \}$, where we think of $P_0$, and $Q_0$ as generators of hyperbolic summand coming from $S^2 \times S^2$, and $P_1$, $Q_1$ represents the other hyperbolic summand. In particular, we have 
\[
P_{i}^2=Q_{i}^2=0,\ P_i.Q_i=1.
\]
Furthermore, let us assume that the $\sigma$ and $\delta$ respectively denote the section and the fiber class of the nucleus $N(2) \subset X \# S^2 \times S^2$. We represent the homology class of $(\sigma + \delta)$, and $\delta$ $P_1$ and $Q_1$ respectively. Going forward, we divide the argument into two cases; let us first assume that $S$ is not a characteristic class. Then it follows from \cite[Lemma 4.1]{Au23}, together with a result of Wall \cite{Wa64} that there is a diffeomorphism of $X \# S^2 \times S^2$ that realizes a automorphism of the intersection form which sends $S$ to 
\[
u v  R_1 + v P_1 + vw Q_1 
\] 
where $u, v, w$ are integers with $v > 0$ and $0 \leq u \leq w$. Let us denote this diffeomorphism by $g$. Now note that if $S$ and $\Phi_p S$ are smoothly isotopic then so are $g(S)$ and $g(\Phi_p(S))$. The same observation also applies to the stabilized surfaces. Hence, it is enough to show that $g(S)$ and $g(\Phi_p(S))$ are not smoothly isotopic, nor are the stabilized copies of them (for any number of stabilization $\leq n$).

To show this, we apply the family adjunction inequality (Theorem~\ref{thm: family adjunction}). First, note that there is a Seiberg--Witten basic class $K$ for $X$ induced from the canonical basic class of $E(2)$. Now since both classes $(\sigma + \delta)$ and $\delta$ are represented torus, by applying the regular adjunction inequality to $(X, K)$, we get
\[
K.(\sigma + \delta)=K.\delta=0.
\]
Moreover, from the adjunction inequality, we also get
\[
\SW(X, K + 2i\delta)=0, \; \mathrm{for} \; i >0.
\]
By a result of Wall \cite{Wa64}, we can assume that $\phi'_p$ on homology splits as a direct summand $\psi^{'}_p \oplus \psi^{''}_p: H^{2}(X_{2p+1}) \oplus H^{2}(S^2 \times S^2) \rightarrow H^{2}(X) \oplus H^{2}(S^2 \times S^2)$, where both components are isomorphisms. Moreover, we can think of $\psi^{'}_p$ being induced from the isomorphism $H^{2}(N(2)_{2p+1}) \rightarrow H^{2}(N(2))$. Now let us define a $\spinc$-structure $\s_p$ on $X_{2p+1} \# S^2 \times S^2$ by
\[
\s_p:= {\psi'}^{-1}_{p}({K + 2p\delta)}.
\]
It follows from the work of Morgan-Mrowka-Szab{\'o} \cite{morgan1997product} that $\s_p$ is a mod 2 basic class for $X_{2p+1} \# S^2 \times S^2$. Now an argument similar to that in the proof of Theorem~\ref{theo: intro sphere from adjunction}, implies that 
\[
\FSW(X \# S^2 \times S^2, K + 2p \delta, g \circ \Phi_p \circ g^{-1})= \FSW(X \# S^2 \times S^2, g \circ \Phi_p \circ g^{-1}) \neq 0.
\]
Now towards contradiction, suppose that after $n'$ (with $0 \leq n' \leq n$) many internal stabilization of $g(S)$ and $g(\Phi_p(S))$, they become smoothly isotopic in $X \# S^2 \times S^2$. Then by family adjunction inequality (Theorem~\ref{thm: family adjunction}), we get
\begin{align*}
2(n + g(S)) -2 \geq 2(n' + g(S)) - 2 \geq (u v  R_1 + v P_1 + vw Q_1).(u v  R_1 + v P_1 + vw Q_1)\\
 + |(K+ 2p Q_1).(u v  R_1 + v P_1 + vw Q_1)|.
\end{align*}
A straight forward computation (using $P_1$ and $Q_1$ are basis of a hyperbolic summand), we get that if
\[
2vp >  2(n + g(S)) -2 - 2 v^2 w + v^2 u + vu|K.R_1|  
\]
there is a contradiction, as desired. The proof for the case when $S$ is characteristic is quite similar, we refer readers to the proof of \cite[Theorem 1.5]{Au23}. The proof that there is a topological isotopy between the surfaces in $C \# S^2 \times S^2$ follows from the work of \cite{OP22}.
  
\end{proof}

\subsection{Equivalent exotic surfaces from non-effective corks}

We now move on to the proof of Theorem~\ref{intro: s1_family_surface_closed}. We will need one more ingredient for this proof, which stems from \cite{KT22, KMT23} manifesting the form of family diagonalization theorem over $T^2$.
For a 4-manifold $X$ with boundary, let  $\Diff(X,\del)$ denote the group of diffeomorphisms which are the identity near the boundary. 

\begin{theo}
\label{thm: rel T2}
Let $X$ be a compact, oriented, smooth 4-manifold with boundary and with $b^+(X)=2$ and $b_1(X)=0$.
Set $Y = \del X$ and suppose that $Y$ is a rational homology 3-sphere.
Let $\fraks$ be a spin$^c$ structure on $X$.
Suppose that there are diffeomorphisms $f_1, f_2 \in \Diff(X,\del)$ such that $[f_1, f_2]$ is isotopic to the identity through $\Diff(X,\del)$, 
and $f_1^\ast, f_2^\ast$ act on $H^+(X)$ as $\diag(-1,1), \diag(1,-1)$ for some choice of basis of $H^+(X)$, and $f_i^\ast \fraks \cong \fraks$.
Then  we have
\[
\frac{c_1(\fraks)^2-\sigma(X)}{8} \leq \delta(Y,\fraks|_Y), 
\]
where $\delta(Y,\fraks|_Y)$ is the monopole Fr\o yshov invariant. 
\end{theo}

\begin{proof}[Proof of Theorem~\ref{intro: s1_family_surface_closed}]
To illustrate the proof, we start with the a specific example. The proof of Theorem~\ref{intro: main_lemma} make it clear to the reader how to generalize the proof.
 
Let $(C,\tau)$ be the Akbulut cork. It was shown in Example~\ref{ex: akbulut cork}, that $(C,\tau)$ admits a filling as $S^1$-family, which consisted of a cobordism $W$ to $\Sigma(2,3,7)$ followed by a filling $W_0$ of $\Sigma(2,3,7)$. Define
\[
X:= C \cup W \cup W_0
\]
We connect sum a copy of $S^2 \times S^2$ on the side of the cork $C$. As described in the proof of  \cite[Theorem 1.9]{KMT23first}, the involution $\tau$ on $\Sigma(2,3,7)$ extends smoothly over $W_0$. We call this diffeomorphism $\tilde{\tau}$. This defines a diffeomorphism $\Phi$ on $X \# S^2 \times S^2$ as follows:
\[
\Phi:= \tilde{\tau} \circ f \circ f_{W_0} : C \# S^2 \times S^2 \cup W \cup W_0 \rightarrow C \# S^2 \times S^2 \cup W \cup W_0.
\]
Note that by definition $\Phi$ maps $C \# S^2 \times S^2$ to itself (although it is not identity on the boundary). Moving forward we will need to use a certain type of Dehn twist on spheres. Since this is the only place in the article, where we use it, we start with a brief definition.

Given a smoothly embedded $(-2)$-sphere $S$ in an oriented 4-manifold $X$, one can get a diffeomorphism $T_S : X \to X$ called the \textit{Seidel Dehn twist} which is supported in a neighborhood of $S$.
Let us recall the construction of $T_S$ below.
A standard reference is \cite[Section 2]{Sei08}.

First, there is a concrete compactly-supported diffeomorphism of $T^\ast S^2$ given as the monodromy of the nodal singulartity $\mathbb{C}^3 \to \mathbb{C}$, $(z_1,z_2,z_3) \mapsto z_1^2+z_2^2+z_3^2$ near the origin of $\mathbb{C}^3$.
We call this diffeomorphism of $T^\ast S^2$ the {\it model Dehn twist}.
It turns out that the model Dehn twist acts as the antipodal map on $S^2$.
Now, given a $(-2)$-sphere $S$ in $X$,
by fixing an identification between $T^\ast S^2$ and a neighborhood of $X$, we implant the model Dehn twist into $X$.
This is the Seidel Dehn twist $T_S : X \to X$ about $S$.
It is clear that the action of $T_S$ on $H_2(X;\Z)$ is given by
\[
(T_S)_\ast(x) = x+(x\cdot[S])[S].
\]
When a $(+2)$-sphere $S$ is given, we can consider a similar diffeomorphism by reversing orientation, and we call it also a Dehn twist about $S$, whose action on $H_2(X;\Z)$ is
\[
(T_S)_\ast(x) = x-(x\cdot[S])[S].
\]

We now consider the Seidel Dehn-twist $T_{\Delta}$ on $C \# S^2 \times S^2$ along the diagonal sphere $\Delta \subset S^2 \times S^2 \subset C \# S^2 \times S^2$. As discussed above, $T_{\Delta}$ is supported in a neighborhood of the $\Delta \subset S^2 \times S^2$. We now claim that $\Delta$ and $\Phi(\Delta)$ are exotic copies of each other. 

To see this, we argue as follows. Towards contradiction we assume that $\Delta$ and $\Phi(\Delta)$ are smoothly isotopic. We consider two diffeomorphisms $T_{\Delta}$ and $T_{\Phi(\Delta)}$, the Seidel Dehn twist along $\Delta$, and $\Phi(\Delta)$ respectively. Observe that the diagonal $\Delta$ (and hence $\Phi(\Delta)$) is a sphere with self-intersection $(+2)$, making it possible to apply the Seidel Dehn twist along these spheres. Note that since Seidel Dehn twists are supported in the neighbourhood of the associated sphere, we can consider $T_{\Delta}$ or $T_{\Phi(\Delta)}$ to be a diffeomorphism of $C \# S^2 \times S^2 \subset X \# S^2 \times S^2$. It follows from our assumption that 
\begin{equation}\label{eq: seidel_1}
T_{\Delta} \simeq T_{\Phi(\Delta)}.
\end{equation}
Moreover, it follows from the definition of the Dehn twist that the following relation holds
\begin{equation}\label{eq: seidel_2}
T_{\Phi(\Delta)} \simeq \Phi \circ T_{\Delta} \circ \Phi^{-1}.
\end{equation}
Combining Relation~(\ref{eq: seidel_1}) and (\ref{eq: seidel_2}) we get
\begin{equation}
T^{-1}_{\Delta} \circ \Phi \circ T_{\Delta} \circ \Phi^{-1} \simeq \mathrm{id}.
\end{equation}
We now observe that the Dehn twist that $T_{\Delta}$ acts as $\mathrm{diag(1,-1)}$ on $H^{+}(X \# S^2 \times S^2)$ while $\Phi$ acts as $\mathrm{diag(-1,1)}$. Let us now define a $\spinc$-structure $\fraks$ on $X \# S^2 \times S^2$ as follows. Let $\s_0$ be the spin structure on $W_0$. Let $\s_1$ be the $\spinc$-structure representing generator $(-1,-1)$ on $W$ which is preserved by the action of $f$. Finally let $\fraks_2$ to be the spin structure on $C \# S^2 \times S^2$. Define
\[
\fraks:= \fraks_2 \cup \fraks_1 \cup \fraks_0.
\]
Moreover, it follows from the construction that both $\Phi$ and $T_{\Delta}$ preserves the $\spinc$-structure $\fraks$. Hence, Theorem~\ref{thm: rel T2} implies that
\[
\frac{c_1(\fraks)^2-\sigma(X)}{8} \leq 0.
\]
This leads to a contradiction since $c_1(\fraks)^2-\sigma(X) > 0$ by construction. This shows that $\Phi(\Delta)$ and $\Delta$ are not smoothly isotopic in $C \# S^2 \times S^2 \subset X \# S^2 \times S^2$ (relative to the boundary of $C$). 

We now have to show that $\Phi(\Delta)$ and $\Delta$ are topologically isotopic in $C \# S^2 \times S^2 \subset X \# S^2 \times S^2$ relative to the boundary of $C$. However, since this is similar to the topological isotopy argument given in the proof of Theorem~\ref{RP2emb}, we omit it. This ends the proof. 

\end{proof}


As corollaries \cref{main real}, we can give the two kinds of exotic surfaces. 
\begin{theo}\label{RP2emb}
 Let $(Y, \tau)$ be a strong cork detected by spin$^c$ preserving family Seiberg--Witten theory.  
    Then, for any contractible bound $C$ of $Y$ which is simple (with respect to $\C P^2$), $C \# \C P^2$ contains a pair of weak equivalent exotic $\R P^2$'s. Moreover, this also gives an exotic pair of $\R P^2$'s in 
    \[
    Z \cup C \# \C P^2
    \]
    for any filling $(Z, f_Z)$ as $S^1$-family of $(Y, \tau)$.
\end{theo}
\begin{theo}\label{real main: link}
      Let $(Y, \tau)$ be a strong cork detected by  $\spinc$-preserving family Seiberg--Witten theory. Then, for any contractible bound $C$ of $Y$ which is simple (with respect to $S^2 \times S^2 $), $C \# (\#_3 S^2 \times S^2)$ contains a pair of embeddings of surface links
      \[
      f, f': S^2_1 \cup S^2_2 \cup S^2_3 \cup S^2_4 \hookrightarrow C \# \#_3 S^2 \times S^2
      \]
  satisfying the following properties: 
  \begin{itemize}
      \item[(i)] $f(S^2_i)$ and $f'(S^2_i)$ are smoothly isotopic for every $i$,
      \item[(ii)] the complements of the images of $f$ and $f'$ are diffeomorphic,
      \item[(iii)] $f$ and $f' $ are exotic, i.e. there is a topological isotopy from $f$ to $f'$ but no smooth isotopy. 
  \end{itemize}
      Moreover, this also gives an exotic pair of four component 2-links in 
    \[
    Z \cup C \# \#_3 S^2\times S^2
    \]
    for any filling $(Z, f_Z)$ as $S^1$-family of $(Y, \tau)$, which has the same properties (i) and (ii). 
\end{theo}

\begin{proof}[Proof of \cref{RP2emb} and \cref{real main: link}]
The proofs of \cref{RP2emb} and \cref{real main: link} are similar. 
Obstructions to smooth isotopies are constructed by \cref{main real}, so we need to ensure the assumptions of \cref{main real} are satisfied. 
\begin{itemize}
    \item For \cref{RP2emb}, we consider the relative connected sum 
    \[
   ( 2 \C P^1 \subset \C P^2) \# (\R P^2_+ \subset S^4) 
    \]
    which is an admissible surface from \cref{ex adm sur}. Here $\R P^2_+$ is the standard embedding of $\R P^2$ into $S^4$ obtained as the quotient with respect to the complex conjugation on $\overline{\C P}^2$. 
    Suppose $C$ is killed by $\C P^2$.  By precomposing with a diffeomorphism on $\C P^2$ (which acts as $(-1)$ in homology) if needed, one can assume the extension is $\id$ on its cohomology. This ensures the assumption of \cref{main real} are satisfied in this case.

    \item For \cref{real main: link}, we consider a surface $W= S^2\times S^2_1 \# S^2\times S^2_2  \# S^2\times S^2_3$ and 
    \begin{align*}
&F :=  S^2 \times \{ N\}_1 \cup S^2 \times \{ S\}_1 \# S^2 \times \{ N\}_2 \cup S^2 \times \{ S\}_2 \# S^2\times \{N\}_3 \cup S^2\times \{S\}_3 \\
&\subset C \# S^2 \times S^2_1 \# S^2 \times S^2_2   \# S^2 \times S^2_3, 
\end{align*}
where $N$ and $S$ are the north and south pole points in $S^2$ and $S^2_i$ represents the $i$-th copy of $S^2 \times S^2$.
This surface link is also admissible from \cref{ex adm sur}. Here we suppose the cork $C$ is killed by $S^2\times S^2$. 
\end{itemize}
In the latter case, we will specify an extension of $C \#3 S^2\times S^2$ later to see each component is relatively and smoothly isotopic.  
This ensures the existence of obstructions to smooth isotopies. 
Now, we consider the existence of topological isotopes. 

Let $(Y, f)$ be a strong cork detected by  $\spinc$-preserving family Seiberg--Witten theory, which is killed by $W= \C P^2 $ or $\#_3 (S^2\times S^2)$, 
i.e. there is an orientation-preserving and homologically trivial diffeomorphism  
\[
\wt{f} : C \#  W \to C \#  W
\]
which gives an extension of the cork action $f$ on $Y$ such that $\wt{f}_* = \id$. In the case of $W=\#_3 S^2\times S^2$, we take a cork $(C, f:\partial C\to C)$ which is killed by one $S^2\times S^2$ and take an extension $C \# 3S^2\times S^2$ of $\tau$ so that $\wt{f}$ is the identity on latter two components $S^2\times S^2_2 \# S^2\times S^2_3$. Then, we claim the embeddings
\[
\iota : S \hookrightarrow C \# W  \text{ and } \wt{f}\circ \iota : S \hookrightarrow C \# W 
\]
are topologically isotopic rel boundary. Note that these surfaces are weakly equivalent each other from the construction. 

We take a topological extension of a cork action $f$ into $C$, denoted by $g$, and consider $g\#  \id$ on $C \# W$. Then the composition $\wt{f}^{-1} \circ (g\#  \id)$ is a relative homologically trivial homeomorphism on $C \# W$. Therefore, one can use \cite{OP22} to get a relative topological isotopy $H_t$ from $\id$ to $\wt{f}^{-1}\circ  (g\#  \id)$. So, $\wt{f} \circ H_t$ gives a topological isotopy from $\wt{f}$ to $g\#  \id$. 
This implies $\iota $ and $\wt{f}\circ \iota $ are topologically isotopic on $C \# W$ rel boundary. This particular provides the topological isotopes between 
\[
\iota : S \hookrightarrow Z\cup C \# W \text{ and } (f_Z\#\wt{f}) \circ \iota : S \hookrightarrow Z\cup C \# W. 
\]

Finally, we prove each component of the four component 2-links is smoothly isotopic. 
With respect to these four components of $F$, we write $F= F_1 \cup F_2 \cup F_3 \cup F_4$. 
Note that we took an extension of the cork action $f$ to $C \# S^2 \times S^2_1$ and extend it by $\id$ on two more $S^2 \times S^2_2 \# S^2 \times S^2_3 $ components, which is denoted by 
\[
\wt{f} : C \# S^2 \times S^2_1 \# S^2 \times S^2_2 \# S^2 \times S^2_3  \to C \# S^2 \times S^2_1 \# S^2 \times S^2_2 \# S^2 \times S^2_3 . 
\] 
We claim $F_i$ and $\wt{f}(F_i)$ are smoothly isotopic for each $i$. From the definition of $\wt{f}$, it is easy to see 
\[
\wt{f} (F_3) = \id (F_3) = F_3 \text{ and } \wt{f} (F_4) = \id (F_4) = F_4.
\]

 Note that the complement of $S^2 \times \{*\}$ in $S^2 \times S^2$ is simply-connected. So, one can see that $C \# W \setminus F_i$ and $C \# W \setminus \wt{f} (F_i)$ are simply-connected as well. 

For $F_1$ and $F_2$, since $F_1$ and $\wt{f}(F_1)$ ($F_2$ and $\wt{f}(F_2)$) are completely contained in $C\# S^2 \times S^2_1 \# S^2\times S^2_2$ and they are not characteristic with simply-connected complements, one can use \cite[Theorem in page 1]{AKMRS19} 
to obtain a smooth isotopy between $F_1$ and $\wt{f}(F_1)$ (resp. $F_2$ and $\wt{f}(F_2)$). This completes the proof. 
  
\end{proof}

We are now in place to prove some of our main theorems. Now \cref{RP2emb1} follows from the Proof of Theorem~\ref{intro: thm RP2 local}.
\begin{proof}[Proof of Theorem~\ref{intro: thm RP2 local}] 
Follows from the Theorem~\ref{RP2emb}.
\end{proof}

\begin{proof}[Proof of Theorem~\ref{RP2emb1}]
Follows from Theorem~\ref{intro: thm RP2 local}.
\end{proof}

\begin{proof}[Proof of Theorem~\ref{S2link}]
Follows from \cref{real main: link}.
\end{proof}

\begin{proof}[Proof of Theorem~\ref{intro: main_lemma}]
With the results proved so far, the proof almost becomes a tautology. For the case of $\mathbb{R}P^2$, this follows from the Theorem~\ref{RP2emb}. For the case of $S^2$, the proof follows from a modification of the proof of Theorem~\ref{intro: s1_family_surface_closed}. Indeed, one replaces the $(C,\tau)$ by any strong cork $(Y,\tau)$ detected by the $\spinc$-preserving family Seiberg--Witten theory (which is killed by $S^2 \times S^2$) and the filling $W \cup W_0$ of $(C, \tau)$ by any filling $(Z,f_Z)$ as $S^1$-family of $(Y,\tau)$. The rest of the proof then implies the desired conclusion. 
\end{proof}

\subsection{Exotic embedding of $S^3$}
We provide equivalent exotic embedding of 3-spheres in small 4-manifolds. 

\begin{proof}[Proof of \cref{theo: intro S3}]
As in proof of Lemma~\ref{lem: sphere from adjunction}, let $X$ be a closed smooth simply-connected oriented 4-manifold given in the definition of simple SW-effective cork (\cref{defi: weakly effective cork}).
Let 
$\Phi : X \#S^2\times S^2 \to X \#S^2\times S^2$
be the exotic diffeomorphism obtained in the proof of Lemma~\ref{lem: sphere from adjunction}.
Recall that $\Phi$ is supported in the interior of $C\#S^2\times S^2$.
Set $Y_i := \Phi^i(Y) \subset C\#S^2\times S^2$ for $i \in \Z$.
We shall prove that $Y_i$ are strongly equivalent to $Y$ but mutually relatively exotic.

By construction, it is clear that $Y_i$ are strongly equivalent to $Y$ and topologically isotopic to $Y$ relative to the boundary of $C\#S^2\times S^2$.
Thus it suffices to prove that $Y_i$ are mutually not smoothly isotopic rel boundary.
Since $Y_i$ and $Y_j$ are smoothly isotopic rel boundary if and only if $Y_{i-j}$ and $Y$ are smoothly isotopic rel boundary, it is enough to see that $Y_i$ is not smoothly isotopic to $Y$ rel boundary for $i\neq0$.

Let $i\neq0$.
Then we have that $\FSW(X\#S^2\times S^2,\fraks',\Phi^i) \neq 0$, where $\fraks'$ is the spin$^c$ structure on $X\#S^2\times S^2$ considered in Lemma~\ref{lem: sphere from adjunction}.
Now, suppose that $Y_i$ is smoothly isotopic to $Y$ rel boundary of $C\#S^2\times S^2$.
This implies that $Y_i$ is smoothly isotopic to $Y$ in $X\#S^2\times S^2$.
After applying a smooth isotopy taking $Y_i$ to $Y$ to $\Phi^i$, we can suppose that $\Phi^i\circ \iota=\iota$, where $\iota : S^3 \hookrightarrow X\#S^2\times S^2$ is the embedding defining $Y$ .
Thus the mapping torus of $\Phi^i$ is the form of the fiberwise connected sum of mapping tori of a diffeomorphism $f_1 : X \to X$ and a diffeomorphism $f_2 : S^2\times S^2\to S^2\times S^2$.
Then we can apply the families gluing \cite[Theorem 9.5]{BK18} to $\Phi^i$ to deduce that $\FSW(X\#S^2\times S^2,\fraks',\Phi^i)=0$ from that $f_2 : S^2\times S^2 \to S^2\times S^2$ is homologically trivial, which follows from the homological triviality of $\Phi^i$.
This is a contradiction and completes the proof.
\end{proof}

\subsection{Algebraic and geometric intersection numbers of embedded spheres}
We end by proving Theorem~\ref{geometric int number} which gives an interesting discrepancy between algebraic and geometric intersection number of embedded 2-spheres.
\begin{proof}[Proof of \cref{geometric int number}]
Let $S$ be a smoothly embedded 2-sphere in $C\#S^2\times S^2$ whose homology class satisfies $[S] = [S^2 \times \{*\} ] + n [\{*\} \times S^2 ] $ for $n>0$, obtained as a graph of degree-$n$ map $S^2\to S^2$.  Now, we consider 
\[
E:= \{*\} \times S^2 \hookrightarrow C \# S^2 \times  S^2 .
\]
Then, we see $E$ is a dual sphere of $S$. 
Define $S'$ to be $\Phi(S)$ where $\Phi$ is defined by a certain commutator relative diffeomorphism on $C \# S^2 \times S^2 $ in the proof of \cref{lem: sphere from adjunction}. 
By isotopy, we assume $E$ intersects $S'$ transversely. Since the algebraic intersection number $E \cdot S'$ is one, we see the number of intersections $E \cap  S'$ is assumed to be $2m+1$ for some $m\in \Z_{\geq 0}$. Note that $m$ cannot be $0$ from Gabai's 4D light bulb theorem \cite{Ga20} \footnote{The original 4D light bulb theorem is sated only for surfaces in closed 4-manifolds, but one can check the proof also works for 4-manifolds with boundary without any essential change.  }  and \cref{lem: sphere from adjunction}.
Now, by composing certain ambient isotopies 
\[
H_t, H'_t : C \# S^2 \times S^2 \times [0,1] \to C \# S^2 \times S^2
\]
rel boundary (with $H_0 = H_1 = \id$), we assume the following: 
\begin{itemize}
    \item the relative diffeomorphism $H_1 \circ \Phi : C \# S^2 \times S^2 \to C \# S^2 \times S^2$ fixes a small closed neighborhood $D:= D^4 \subset C \# S^2 \times S^2$. 
    \item all intersection points of the image of $H'_1\circ S': S^2 \hookrightarrow C \# S^2 \times S^2$ and $E$ lie the interior of $D$. 
\end{itemize}
Now, we remove pairs of negative and positive intersection points by the surgery along these points, adding 2-dimensional 1-handles, to obtain a new embedding written by $H'_1\circ S' \#_m T^2$, which is smoothly isotopic to $m$-internal stabilizations of $H'_1\circ S'$ rel boundary. Then, we see that $H'_1\circ E$ is a dual 2-sphere for $H'_1\circ S' \#_m T^2$ and $H'_1\circ S'$. 
 Now, we see 
 $H \circ \Phi(H'_1\circ S' \#_m T^2)$ and $H'_1\circ S'  \#_m T^2$ are smoothly isotopic again from Gabai's 4D light bulb theorem \cite{Ga20}.
 Now, we apply the family adjunction inequality as in the proof of \cref{lem: surfaces from adjunction} to see 
 \[
 2m -2 = 2 g(H'_1\circ S'  \#_m T^2) -2 \geq [H'_1\circ S'  \#_m T^2]^2 =2n . 
 \]
 This completes the proof. 
\end{proof}

\appendix

\bibliographystyle{amsalpha}
\bibliography{tex}

\end{document}